\documentclass[12pt]{amsart}%
\usepackage{amsfonts}
\usepackage{amsmath}
\usepackage{amssymb}
\usepackage{amscd}
\usepackage{graphicx}%
\providecommand{\U}[1]{\protect \rule{.1in}{.1in}}
\newtheorem{theorem}{Theorem}
\newtheorem{lemma}{Lemma}
\newtheorem{proposition}{Proposition}
\newtheorem{remark}{Remark}

\newtheorem{definition}{Definition}
\newtheorem{corollary}{Corollary}

\numberwithin{equation}{section}
\begin{document}
\title[Rigidity of Sasaki-Ricci Solitons]{Transverse Rigidity of Shrinking Sasaki-Ricci Solitons}
\author{$^{\ast}$Shu-Cheng Chang}
\address{$^{\ast}$Department of Mathematics, National Taiwan University, Taipei 10617,
Taiwan and Shanghai Institute of Mathematics and Interdisciplinary Sciences,
Shanghai, 200433, China}
\email{scchang@math.ntu.edu.tw}
\author{$^{\dag}$Fengjiang Li}
\address{$^{\dag}$Mathematical Science Research Center, Chongqing University of
Technology, 400054, Chongqing, P.R. China}
\email{fengjiangli@cqut.edu.cn}
\author{$^{\dag \dag}$Chien Lin}
\address{$^{\dag \dag}$Department of Mathematics, National Taiwan Normal University,
Taipei, 116059, Taiwan}
\email{chienlin@ntnu.edu.tw}
\author{$^{\dag \dag \dag}$Hongbing Qiu}
\address{$^{\dag \dag \dag}$School of Mathematics and Statistics, Wuhan University, Wuhan
430072, China }
\email{hbqiu@whu.edu.cn}
\thanks{$^{\ast}$Shu-Cheng Chang is partially supported by Startup Foundation for
Advanced Talents of the Shanghai Institute for Mathematics and
Interdisciplinary Sciences (No.2302-SRFP-2024-0049). $^{\dag}$Fengjiang Li is
partially supported by NSFC (No.12301062 and 12571050) and NFSCQ
(No.STB2024NSCQ-MSX0537). $^{\dag \dag}$Chien Lin is partially supported by
NSTC 114-2115-M-003-002-MY2.}
\date{}

\begin{abstract}
In this paper, we study several properties of Sasaki-Ricci solitons as
singularity models of the Sasaki-Ricci flow. First, we establish several
fundamental equations for Sasaki-Ricci solitons, which enable us to derive
potential estimates and prove the positivity of the scalar curvature. Then we
present two criteria for the transverse rigidity of Sasaki-Ricci solitons. As
essential applications, we prove that any low-dimensional Sasaki-Ricci soliton
with constant scalar curvature must be Sasaki-Einstein, and that any
Sasaki-Ricci soliton with harmonic Weyl tensor is a finite quotient of the sphere.

\end{abstract}
\keywords{Sasaki-Ricci Flow, Sasaki-Ricci Soliton, Sasaki-Einstein Metric, Sasakian
Metric of Constant Scalar Curvature, Rigidity.}
\subjclass{Primary 53E50, 53C25; Secondary 53C12, 14E30.}
\maketitle

\section{Introduction}

Inspired by Eells and Sampson's work on harmonic map heat flow in 1960's,
Hamilton engaged in studying the deformation of Riemannian metrics along the
so-called Ricci flow. Since publishing the seminal paper \cite{ha82} in 1982,
the Ricci flow method has become a natural and essential tool not only to
attack the Poincar\'{e} conjecture and Thurston's geometrization conjecture
(see\cite{pe02}, \cite{pe03.1}, \cite{pe03.2}, \cite{cz06}) but to unveil the
existence of canonical metrics including Einstein metrics, metrics of constant
scalar curvature, Ricci solitons, extremal metrics (see\cite{ha82},
\cite{ha86}, \cite{ha88}, etc.). However, the metrics evolved by the Ricci
flow usually come up with the possible formation of singularities which happen
in the highly-curved area. Ricci solitons often arise as the blow-up limits of
singularity dilations along the Ricci flow. More precisely, Hamilton
\cite{ha95} showed that any maximal solution of Type I, II or III to the Ricci
flow satisfying some mild condition (like the injectivity radius condition or
bounded nonnegative curvature operator) admits a sequence of singularity
dilations of the solution which converges to a singularity model of the
corresponding Type in the $C_{loc}^{\infty}$-topology. And, under the proper
assumption, Hamilton \cite{ha93}, Enders-M\"{u}ller-Topping \cite{emt11},
Chen-Zhu \cite{cz00} and Cao \cite{cao97} demonstrated that the singularity
models of Type I, II or III shall be the shrinking, steady or expanding Ricci
solitons respectively.

A \textsl{Ricci soliton} is meant to be the triple $\left(  M,g,X\right)  $
which consists of a manifold $M$, a Riemannian metric $g$, and a vector field
$X$ on $M$ and satisfies the soliton equation%
\[
Ric+\frac{1}{2}L_{X}g+\lambda g=0
\]
for a constant $\lambda$. In addition, when the vector field $X$ is a gradient
field of some smooth function $f$, we say such Ricci soliton $\left(
M,g,f\right)  \doteqdot \left(  M,g,\nabla f\right)  $ a \textsl{gradient Ricci
soliton} with the potential $f$. A Ricci soliton is called of
\textsl{shrinking} (resp. \textsl{steady} or \textsl{expanding}) whenever the
constant $\lambda$ is negative (resp. zero or positive). Due to being the
singularity models to the sequences of dilations, the classification of
(complete gradient) Ricci solitons is a fundamental and central issue in the
framework of the Ricci flow. We refer the interested reader to the survey
\cite{cao10} and references therein for the background. From the point\ of the
soliton equation, people often view the Ricci solitons as a generalization of
Einstein metrics. So it is intriguing to ask whether there is non-Einstein
Ricci soliton. In the compact case, Perelman \cite{pe02} asserted that any
Ricci soltion must be of gradient type. Combining with Hamilton's results
\cite{ha88}\cite{ha95} and Ivey's one \cite{i93}, we see that any compact
steady or expanding Ricci soliton is Einstein, and compact shrinking Ricci
solitons of dimension at most three are Einstein. And the first non-Einstein
compact shrinking Ricci soliton was constructed by Koiso \cite{koi90} and Cao
\cite{cao96} independently on the one-point blow-up $Bl_{1}%
\mathbb{C}
\mathbb{P}^{2}$ of the complex projective surface. As for the noncompact case,
there are many examples of non-Einstein gradient Ricci solitons such as the
complete rotationally symmetric steady or expanding gradient Ricci soliton due
to Bryant \cite{br05}, the complete rotationally symmetric shrinking gradient
Ricci solitons attributed to Kotschwar \cite{kot08}, and etc. For the two- and
three-dimensional gradient Ricci soltion had been classified to some extent
(see \cite{ha88}, \cite{br05}, \cite{i92}, \cite{ra13}, \cite{bm15},
\cite{cho23}).

In the authors' own flavor, the rest of this paper shall focus on the complete
gradient Ricci solitons of shrinking type. One reason is that a Riemannian
manifold with positivity/nonnegativity curvature is usually endowed with much
better rigidity; and another comes from Fano Yau-Tian-Donaldson conjecture:
the existence of Fano K\"{a}hler-Einstein metrics is equivalent to the
K-polystability (see\cite{cds15.1}, \cite{cds15.2}, \cite{cds15.3},
\cite{t15}, \cite{sz16}, \cite{ds16}, \cite{csw18}, \cite{bbj21},
\cite{zha24}). As aforementioned in the last paragraph, Hamilton classified
all two-dimensional complete shrinking gradient Ricci solitons: either
isometric to the standard Euclidean plane $%
\mathbb{R}
^{2}$, or the round $2$-sphere $\mathbb{S}^{2}$, or the real projective plane
$%
\mathbb{R}
\mathbb{P}^{2}$. For the complete three-dimensional gradient shrinking Ricci
solitons, Ivey \cite{i93}, Perelman \cite{pe03.1}, Ni-Wallach \cite{nw08}, and
Cao-Chen-Zhu \cite{ccz08} demonstrated that it shall be isometric to a finite
quotient of either the round $3$-sphere $\mathbb{S}^{3}$, or the shrinking
Gaussian soliton $%
\mathbb{R}
^{3}$, or the round cylinder $\mathbb{S}^{2}\times%
\mathbb{R}
$. Hence, it is natural to investigate the classification of complete
four-dimensional gradient Ricci solitons. To our limited knowledge, it is
actually far from the complete classification although there is substantial
progress toward this objective in recent years. Fortunately, in the scheme of
K\"{a}hler geometry, such classification problem about the K\"{a}hler-Ricci
shrinker surfaces, i.e. complete two-dimensional shrinking gradient
K\"{a}hler-Ricci solitons, has been done by Yu Li and Bing Wang \cite{lw23}
last year. Actually, they manifested any two-dimensional complete shrinking
K\"{a}hler-Ricci soliton admits bounded scalar curvature which implies the
bounded Riemannian curvature by Munteanu-Wang \cite{mw15}; then, combining
this with the works of the existence and uniqueness about the compact
K\"{a}hler-Ricci shrinker surfaces (cf. \cite{ty87}, \cite{si88}, \cite{t90},
\cite{koi90}, \cite{cao96}, \cite{wz04}, \cite{bm87}, \cite{tzhu00},
\cite{tzhu02}) and the noncompact case (cf. \cite{fik03}, \cite{cds24},
\cite{ccd22}, \cite{bccd24}), it could be derived that any K\"{a}hler-Ricci
shrinker surface must be isometrically biholomorphic to one of the following:
\[
\left \{
\begin{array}
[c]{cl}%
\left(  \mathrm{1}\right)  & \mathrm{Closed}\text{ \textrm{del Pezzo
surfaces}}\\
\left(  \mathrm{2}\right)  & \text{\textrm{Gaussian soliton on} }%
\mathbb{C}
^{2}\\
\left(  \mathrm{3}\right)  & \text{\textrm{FIK shrinker on} }Bl_{p}%
\mathbb{C}
^{2}\\
\left(  \mathrm{4}\right)  & \left(
\mathbb{C}
\mathbb{P}^{1}\times%
\mathbb{C}
,g_{sta}\right) \\
\left(  \mathrm{5}\right)  & \text{\textrm{BCCD shrinker on} }Bl_{p}\left(
\mathbb{C}
\mathbb{P}^{1}\times%
\mathbb{C}
\right)
\end{array}
\right.  .
\]

Besides the classification of K\"{a}hler-Ricci shrinker surfaces, which
benefits from techniques in K\"{a}hler geometry, an attempt to figure out the
structure about complete four-dimensional gradient Ricci solitons is to impose
some geometric restrictions on reducing its complication to be manipulated
availably. Before giving a very brief survey on the related results, we
primarily present the notion of rigidity introduced by Petersen and Wylie
\cite{pw09}. As a mixed extension of Einstein metrics and Gaussian solitons,
they define a gradient Ricci soliton $\left(  M,g,f\right)  $ to be
\textsl{rigid} if it is isometric to a quotient space $N\times_{\Gamma}%
\mathbb{R}
^{k}$, where $N$ is an Einstein manifold and $f=\frac{\lambda}{2}\left \vert
x\right \vert ^{2}$ on the Euclidean factor. Here $\Gamma$ acts freely on $N$
and by orthogonal transformations on $%
\mathbb{R}
^{k}$. With the help of the fact that the fundamental group of a shrinking
gradient Ricci soliton is finite \cite{wy08}, a shrinking gradient Ricci
soliton is rigid if it is isometric to a finite quotient space. In terms of
the language, the aforementioned results says that every complete shrinking
gradient Ricci soliton of dimension less than four is rigid. For the
high-dimensional case of the classification, the difficulty mostly arises from
the presence of the Weyl tensor. It was verified that any complete locally
conformally flat $n$-dimensional shrinking gradient Ricci soliton must be
rigid; accurately it is isometric to a finite quotient of either
$\mathbb{S}^{n}$, $\mathbb{S}^{n-1}\times%
\mathbb{R}
$, or $%
\mathbb{R}
^{n}$ (see \cite{cwz11}, \cite{elnm08}, \cite{ms13}, \cite{nw08}, \cite{pw10},
\cite{zha09.1} and the references therein). More generally,
Fern\'{a}ndez-L\'{o}pez and Garc\'{\i}a-R\'{\i}o \cite{flgr11}, and Munteanu
and Sesum \cite{ms13} showed that, under the hypothesis of carrying the
harmonic Weyl tensor, every complete $n$-dimensional shrinking gradient Ricci
soliton shall be rigid. More references about putting various conditions on
the Weyl-type tensors could be referred to \cite{cz23}. On the other hand,
Naber \cite{n10} proved that a complete noncompact four-dimensional shrinking
gradient Ricci soliton with bounded nonnegative curvature operator, then it is
isometric to $%
\mathbb{R}
^{4}$ or to a finite quotient of either $\mathbb{S}^{2}\times%
\mathbb{R}
^{2}$ or $\mathbb{S}^{3}\times%
\mathbb{R}
$. It can be found that all of these solitons are rigid. On the contrary, as
mentioned before, the first non-rigid compact shrinking gradient
K\"{a}hler-Ricci solitons of complex dimension two were constructed by Koiso
\cite{koi90} and Cao \cite{cao96} independently. Therefore, it is imperative
to characterize the rigidity of a gradient Ricci soliton. Such
characterization was done by Eminenti-La Nave-Mantegazza \cite{elnm08} for
compact gradient Ricci solitons and Petersen-Wylie \cite{pw09} for general
case, that is a gradient Ricci soliton is rigid if and only if it admits
constant scalar curvature and it is radially flat which means that the
sectional curvature along the direction of $\nabla f$ vanishes%
\[
K\left(  \cdot,\nabla f\right)  =0.
\]
It is worth to note that the radially flat condition is redundant for compact
case. Motivated by this characterization, people was engaged in the
classification of complete shrinking gradient Ricci solitons under the
assumption of constant scalar curvature. By the endeavors of Petersen-Wylie
\cite{pw09} and Fern\'{a}ndez\textbf{-}L\'{o}pez - Garc\'{\i}a-R\'{\i}o
\cite{flgr16}, it has been shown that the possible values of the scalar
curvature of a complete $n$-dimensional shrinking gradient Ricci solitons with
constant scalar curvature must lie in the set $\left \{  k\lambda \right \}
_{k=0}^{n}$, the shrinking soltion must be rigid when the scalar curvatures
take the first value $0$ and the last two values $\left(  n-1\right)  \lambda$
and $n\lambda$, and it can not achieve the value $\lambda$ as the scalar
curvature. More properties can be found in their papers. Recently, Xu Cheng
and Detang Zhou \cite{cz23} fill in the last piece of puzzle to classify all
complete gradient four-dimensional Ricci solitons of constant scalar
curvature, i.e. all of them are rigid.

In 1962, Sasaki and Hatakeyama \cite{sh62} introduced the concept of "Sasakian
structures" under the name of "normal contact metric structures". With the
influence of the works \cite{bw58}, \cite{hat63}, \cite{ta63}, \cite{ou85},
\cite{bg00}, \cite{bg08} and etc., over the years, Sasakian manifolds are
served as an odd-dimensional analogue of K\"{a}hler manifolds. For instance,
the K\"{a}hler cone of a Sasaki-Einstein $5$-manifold is a Calabi-Yau
$3$-fold. It provides intriguing examples of the AdS/CFT correspondence.
Motivated by the content as precedes, it is very natural to ask whether there
are assembling outcomes in the category of Sasakian geometry. That is if there
exist the characterization of the rigidity of Sasaki-Ricci solitons and its
(rough) classification in low-dimensional Sasakian manifolds. Our goal of the
series of papers is to address the classification problem of Sasaki-Ricci
solitons of dimension five as the Sasakian version of Li-Wang's paper
\cite{lw23}. In this paper, we primarily focus on the first question and the
second one will be discussed in the future study. As a typical generalization
of Sasaki-Einstein metrics, the triple $\left(  M^{2n+1},g,X\right)  $ is
called a \textsl{(shrinking) Sasaki-Ricci soliton} \textsl{with the
Hamiltonian potential} $\psi$ \textsl{with respect to the Hamiltonian
holomorphic vector field} $X$ if it satisfies
\[
Ric^{T}+\frac{1}{2}L_{X}g^{T}-\left(  2n+2\right)  g^{T}=0.
\]
The precise definition of Sasaki-Ricci solitons could be referred to the
preliminary section.

Similar to the result that all $2$-dimensional complete nonflat gradient
shrinking Ricci solitons must be compact and Einstein in \cite{ha95}, we have
the fact that any complete Sasaki-Ricci soliton of dimension $3$ shall reduce
to the compact Sasaki-Einstein one. This could be observed from Theorem 4.9 in
\cite{cc09} and Theorem 1 in \cite{be01} with our result about the positivity
of scalar curvature $R$ of Sasaki-Ricci solitons as below; more precisely, by
taking $\epsilon=2$, $A=\frac{R+2}{2}>1$, $B=0$ in their notations, Theorem
4.9 in \cite{cc09} which furnishes us a complete adapted metric $g_{\epsilon}$
with Ricci curvature bounded from below by a positive constant ensures that
the underlying manifold is compact, and then the result is easily deduced from
Theorem 1 in \cite{be01}. Consequently, it is intriguing to ask whether such
fact still hold for complete Sasaki-Ricci solitons of dimension $5$ or not. As
an application of the criterion in the following, it is true in the situation
of constant scalar curvature.

The first result of the paper is about the quantification of possible scalar
curvatures of a Sasaki-Ricci soliton with constant scalar curvature as below.

\begin{proposition}
If $\left(  M^{2n+1},g,X\right)  $ is a complete Sasaki-Ricci soliton with
constant scalar curvature $R$ and the Hamiltonian potential $\psi$ with
respect to $X$, then $R$ must lie in the finite set $\left \{  \left(
2n-1\right)  k+\left(  2n+1\right)  \right \}  _{k=1}^{2n+1}$.
\end{proposition}

The method we employed is about the technique in the isoparametric theory
developed by Qi-Ming Wang \cite{wa87} similarly in \cite{flgr16}.

Subsequently, we present our main results about the characterization of the
transverse rigidity of Sasaki-Ricci solitons.

\begin{theorem}
If $\left(  M,g,X\right)  $ is a complete Sasaki-Ricci soliton with the
Hamiltonian potential $\psi$ with respect to $X$, then the following are
equivalent: \newline(i) $\left(  M,g,X\right)  $ is transversely rigid.
\newline(ii) The scalar curvature is constant and it is transversely radially
flat, i.e. \newline$R\left(  \cdot,\nabla^{T}\psi \right)  \nabla^{T}\psi$
vanishes along the horizontal direction. \newline(iii) The scalar curvature is
constant and $g\leq Ric\leq2ng$.
\end{theorem}

On the other hand, under the assumption of constant scalar curvature, there is
another type of criterion related to the rank of Ricci curvature operator.

\begin{theorem}
A complete Sasaki-Ricci soliton $\left(  M,g,X\right)  $ of constant scalar
curvature is Sasaki-Einstein if and only if the rank of the operator $Ric-g$
is constant.
\end{theorem}

As an imperative application, we derive the triviality of low-dimensional
Sasaki-Ricci solitons when the scalar curvature is constant.

\begin{corollary}
\label{CSC}Any complete Sasaki-Ricci soliton of dimension at most seven with
constant scalar curvature must be Sasaki-Einstein.
\end{corollary}

Combing with the result \cite{bre10} of S. Brendle, we are able to obtain the following:

\begin{corollary}
Any complete Sasaki-Ricci soliton of dimension five or seven with constant
scalar curvature and positive isotropic curvature (PIC) must be a finite
quotient of the unit sphere. Furthermore, if it is of nonnegative isotropic
curvature instead of PIC, then it is locally symmetric.
\end{corollary}

\begin{remark}
In 2009, Futaki, Ono and Wang \cite{fow09} demonstrated that there is a
Sasaki-Ricci soltion on any compact transverse Fano toric Sasakian manifold
with vanishing first Chern class of the contact bundle; moreover, it carries a
Sasaki-Einstein metric if and only if the Sasaki-Futaki invariant equals zero.
For a compact toric Sasakian manifold, their conclusion says that any
Sasaki-Ricci soltion with constant scalar curvature admits a Sasaki-Einstein
metric. So our corollary could be regarded as a generalization of such result
in low-dimensional cases.
\end{remark}

Last, we consider the rigidity of Sasaki-Ricci solitons under the assumption
of harmonic Weyl tensor.

\begin{theorem}
\label{harWT}If $\left(  M^{2n+1},g,X\right)  $ be a complete Sasaki-Ricci
soliton with the Hamiltonian potential $\psi$ with respect to $X$ with
harmonic Weyl tensor, then it is a finite quotient of the sphere $S^{2n+1}$.
\end{theorem}

This paper is organized as follows. In Section 2, we review some preliminaries
in Sasakian geometry. In Section 3, we establish a fundamental estimate for
the Hamiltonian potential of a complete noncompact Sasaki-Ricci soliton, under
the assumption that the scalar curvature is nonnegative. In Section 4, we
verify that this assumption indeed holds, namely, that the scalar curvature of
any complete noncompact Sasaki-Ricci soliton is nonnegative, and in fact
strictly positive by the maximum principle. Section 5 is devoted to the study
of the possible values of the scalar curvature of a Sasaki-Ricci soliton,
along with two criteria for transverse rigidity and several applications.
Finally, in the last section, we characterize Sasaki-Ricci solitons with
harmonic Weyl tensor.

\textbf{Acknowledgements.} Part of this project was carried out during the
third named author's visit to the NCTS to whom he would like to express the
gratitude for the warm hospitality.

\section{Preliminary}

In this section, we will review some fundamental notion for Sasakian
manifolds. The reader is referred to \cite{bg08}, \cite{bl10}, \cite{fow09},
\cite{zha11} and the references therein for more details. A Riemannian
$(2n+1)$-manifold $(M,g)$ is said a Sasakian manifold if\ the metric cone
$(C(M),\overline{g},\overline{J})\triangleq(\mathbb{R}^{+}\times
M\mathbf{,\ }dr^{2}+r^{2}g,\overline{J})$ is K\"{a}hler. Such manifold is
canonically endowed with the so-called Sasakian structure (or normal contact
metric structure originally adopted by Sasaki and Hatakeyama) $\left(
\phi,\xi,\eta,g\right)  $ which is defined by%
\[
\phi \doteqdot \left \{
\begin{array}
[c]{ll}%
\overline{J} & \mathrm{on}\text{ }D\\
0 & \mathrm{on}\text{ }L_{\xi}%
\end{array}
\right.  ,\text{ }\xi \doteqdot \left(  \overline{J}\frac{\partial}{\partial
r}\right)  |_{M},\text{ }\mathrm{and}\text{ }\eta \doteqdot g\left(  \cdot
,\xi \right)  .
\]
Here $D$ and $L_{\xi}$ denote the contact bundle $\ker \eta$ and the trivial
real line bundle induced by the Reeb vector field $\xi$ respectively. In this
case, the Sasakian manifold $M$ is viewed as the section $\left[  r=1\right]
$ of the K\"{a}hler cone $C(M)$. Conversely, we could also recover the
K\"{a}hler structure of the metric cone $(C(M),\overline{g},\overline{J})$
from the Sasakian structure $\left(  \phi,\xi,\eta,g\right)  $; actually,
there is a one-to-one correspondence between them (cf. \cite{bg08} and
\cite{bl10}). It is clear that any Sasakian manifold $M$ carries a foliated
structure $\mathcal{F}=$ $\mathcal{F}_{\xi}$ with integrable line bundle
$L_{\xi}$, usually called the Reeb foliation. Therefore, the transverse
geometry which comes from the foliation theory plays a significant role in the
study of Sasakian manifolds. In this section, we always assume $M$ is a
Sasakian manifold. There exists an exact sequence of vector bundles over $M$
\[
0\longrightarrow L_{\xi}\longrightarrow TM\overset{\pi}{\longrightarrow
}v\left(  \mathcal{F}\right)  \longrightarrow0
\]
where the quotient bundle $v\left(  \mathcal{F}\right)  \doteqdot$ $TM/L_{\xi
}$ is said the normal bundle of the Reeb foliation. With the help of the
orthogonal decomposition $TM=$ $D\oplus L_{\xi}$ with respect to the Sasakian
metric $g$, we are able to identify the normal bundle $v\left(  \mathcal{F}%
\right)  $ with the contact bundle $D$ by
\[%
\begin{array}
[c]{clll}%
\sigma: & v\left(  \mathcal{F}\right)  & \longrightarrow & D\\
& Y+L_{\xi} & \longmapsto & Y-\eta \left(  Y\right)  \xi
\end{array}
.
\]

By this identification $v\left(  \mathcal{F}\right)  \simeq D$, one can define
the transverse Levi-Civita connection (or Bott connection) $\nabla^{T}$ on the
normal bundle $v\left(  \mathcal{F}\right)  $ over $M$ by%
\[
\nabla_{X}^{T}\overline{Y}\doteqdot \left \{
\begin{array}
[c]{ll}%
\pi \left(  \nabla_{X}\sigma \left(  \overline{Y}\right)  \right)  &
\mathrm{if}\text{ }X\in \Gamma \left(  D\right) \\
\pi \left(  \left[  \xi,\sigma \left(  \overline{Y}\right)  \right]  \right)  &
\mathrm{if}\text{ }X=\xi
\end{array}
\right.
\]
where $\overline{Y}$ is meant to be the class $Y+L_{\xi}$ and $\nabla$ is the
Levi-Civita connection of $(M,g)$. As in the Riemannian case, it could be
asserted that the transverse Levi-Civita connection is torsion-free and
compatible with the transverse metric $g^{T}\doteqdot \frac{1}{2}d\eta \left(
\cdot,\phi \left(  \cdot \right)  \right)  $ by recognizing $v\left(
\mathcal{F}\right)  $ as $D$. Moreover, this leads to the transverse
Riemannian curvature operator $Rm^{T}$ and a variety of transverse curvature
tensors as well as in the usual sense.

\begin{remark}
As mentioned in \cite{zha11}, it is worth to note that, under the deformations
of Sasakian structure, the contact bundle $D$ would vary while the normal
bundle $v\left(  \mathcal{F}\right)  $ stays the same. Nonetheless, because
the calculations in the content do not be involved with the deformations of
Sasakian structure, the contact bundle $D$ shall be identified with the normal
bundle $v\left(  \mathcal{F}\right)  $ in the paper.
\end{remark}

In virtue of the complex Reeb foliation induced by the corresponding
K\"{a}hler cone, it reveals that a Sasakian manifold carries the transverse
complex geometric structure. We shall present some basic materials about it below.

\begin{definition}
An endomorphism $J$ of the normal bundle $v\left(  \mathcal{F}\right)  $ is
called a transverse almost complex structure if
\[
J^{2}=-Id_{v\left(  \mathcal{F}\right)  }.
\]

\end{definition}

As aforementioned, the K\"{a}hler cone of a Sasakian manifold $M^{2n+1}$ is
endowed with the complex Reeb foliation defined by the holomorphic vector
field $\left(  \frac{\partial}{\partial r}-i\overline{J}\frac{\partial
}{\partial r}\right)  $ on $C(M)$. It means that there is a complex foliated
atlas $\left \{  \left(  \widetilde{U}_{\alpha},\widetilde{\phi}_{\alpha
}=\left(  z_{j}^{\alpha}\right)  _{j=1}^{n+1}\right)  \right \}  _{\alpha
\in \Lambda}$ on $C(M)$ such that the transition functions
\[
\widetilde{\tau}_{\beta \alpha}:\widetilde{f}_{\alpha}\left(  \widetilde
{U}_{\alpha \beta}\right)  \longrightarrow \widetilde{f}_{\beta}\left(
\widetilde{U}_{\alpha \beta}\right)
\]
of the local submersions $\widetilde{f}_{\alpha}\doteqdot p_{2}\circ
\widetilde{\phi}_{\alpha}$ and $\widetilde{f}_{\beta}\doteqdot p_{2}%
\circ \widetilde{\phi}_{\beta}$ on the overlap of foliated charts
$\widetilde{U}_{\alpha \beta}\doteqdot \widetilde{U}_{\alpha}\cap \widetilde
{U}_{\beta}$ are biholomorphic where $p_{2}:%
\mathbb{C}
\times%
\mathbb{C}
^{n}\longrightarrow%
\mathbb{C}
^{n}$ denotes the projection onto the second factor. Therefore, we have an
underlying transverse holomorphic structure on $\mathcal{F}$, i.e. given an
open covering $\left \{  U_{\alpha}\right \}  _{\alpha}$ and local submersions
$f_{\alpha}:U_{\alpha}\longrightarrow%
\mathbb{C}
^{n}$ which are induced from the Reeb foliation $\mathcal{F}$, there exist
biholomorphisms $\tau_{\beta \alpha}:f_{\alpha}\left(  U_{\alpha \beta}\right)
\longrightarrow f_{\beta}\left(  U_{\alpha \beta}\right)  $ for all indices
$\alpha$,$\beta$ by taking $U_{\alpha}\doteqdot \widetilde{U}_{\alpha}\cap M$,
$f_{\alpha}\doteqdot \widetilde{f}_{\alpha}|_{M}:U_{\alpha}\longrightarrow%
\mathbb{C}
^{n}$ and $\tau_{\beta \alpha}=\widetilde{\tau}_{\beta \alpha}$. It could be
found that this transverse holomorphic structure naturally defines a
transverse almost complex structure $J$ on $v\left(  \mathcal{F}\right)  $ and
then the complexification $v\left(  \mathcal{F}\right)  ^{%
\mathbb{C}
}\doteqdot v\left(  \mathcal{F}\right)  \otimes_{%
\mathbb{R}
}%
\mathbb{C}
$ is decomposed as%
\[
v\left(  \mathcal{F}\right)  ^{%
\mathbb{C}
}=v\left(  \mathcal{F}\right)  ^{1,0}\oplus v\left(  \mathcal{F}\right)
^{0,1}%
\]
where $v\left(  \mathcal{F}\right)  ^{1,0}$ and $v\left(  \mathcal{F}\right)
^{0,1}$ are the eigenspaces of $J$ corresponding to the eigenvalues $i$ and
$-i$ respectively. We say that a vector field $X$ on $M$ is foliate if
$L_{\xi}$ is preserved by $X$.

\begin{definition}
\bigskip A foliated vector field $X$ on $M$ is called a transverse (real)
holomorphic vector field if%
\[
\left[  \overline{X},J\overline{Y}\right]  =J\overline{\left[  X,Y\right]  }%
\]
for any vector field $Y$ on $M$.
\end{definition}

As usual, these concepts could be easily generalized to the complex version.
For example, a transverse complex holomorphic vector field $Z$ is meant to be
a foliated complex vector field which satisfies the last identity and the
class $Z+%
\mathbb{C}
L_{\xi}\in v\left(  \mathcal{F}\right)  ^{1,0}$.\ Now, in terms of these
terminologies, we are able to introduce the Hamiltonian holomorphic vector
fields subsequently.

\begin{definition}
A transverse complex holomorphic vector field $X$ is called a Hamiltonian
holomorphic vector field if there is a complex-valued basic function $\psi$
such that
\begin{equation}
\left \{
\begin{array}
[c]{l}%
\psi=i\eta \left(  X\right) \\
\iota_{X}\omega^{T}=i\overline{\partial}_{B}\psi
\end{array}
\right.  . \label{pr1}%
\end{equation}
Here $\iota_{X}$ denotes the contraction with $X$ and $\omega^{T}=\frac{1}%
{2}d\eta=g^{T}\left(  \phi \left(  \cdot \right)  ,\cdot \right)  $.
\end{definition}

\begin{remark}
Given a complex-valued basic function $\psi$, it is apparent to see that there
is a unique Hamiltonian holomorphic vector field $X$ which fulfills the
conditions (\ref{pr1}).
\end{remark}

Subsequently, we give the definition of the Sasaki-Ricci solitons (someone
calls it the transverse K\"{a}hler-Ricci solitons).

\begin{definition}
The triple $\left(  M,g,X\right)  $ is called a Sasaki-Ricci soliton (with the
Hamiltonian potential $\psi$) with respect to the Hamiltonian holomorphic
vector field $X$ if it satisfies%
\[
Ric^{T}+\frac{1}{2}L_{X}g^{T}-\left(  2n+2\right)  g^{T}=0
\]
where $X$ is the Hamiltonian holomorphic vector field satisfying (\ref{pr1}).
\end{definition}

By the way, one can consider more general definition for Sasaki-Ricci
solitons, that means the coefficient of $g^{T}$ is allowed to be any real
number as usual; however, in this paper, Sasaki-Ricci solitons are regarded as
the generalization of Sasaki-Einstein metrics. Last, the transverse rigidity
of the Sasaki-Ricci solitons is described as follows.

\begin{definition}
A Sasaki-Ricci soliton $\left(  M^{2n+1},g,X\right)  $ (with the Hamiltonian
potential $\psi$) with respect to the Hamiltonian holomorphic vector field $X$
is said to be transversely rigid if it is isometric to a quotient $N\times_{G}%
\mathbb{C}
^{k}$ of the product of a Sasaki $\eta$-Einstein manifold $N$ and a complex
Euclidean space $%
\mathbb{C}
^{k}$ where $\psi=n\left \vert z\right \vert ^{2}$ for $z\in%
\mathbb{C}
^{k}$. Here the finite group $G$ acts on $N$ freely and on $%
\mathbb{C}
^{k}$ by unitary transformation.
\end{definition}

\section{Potential Estimates of the Sasaki-Ricci Solitons}

Before establishing the nonnegativity, and in fact the positivity, of the
scalar curvature for Sasaki-Ricci solitons in the next section, it is natural
to first present the Sasaki analogue of the potential function estimates as in
\cite{cd10} (or \cite{hm11}).

\begin{lemma}
\label{potential estimate}If $\left(  M^{2n+1},g,X\right)  $ is a complete
noncompact Sasaki-Ricci soliton with nonnegative scalar curvature and the
corresponding Hamiltonian function (potential) $\psi$ with respect to $X$,
then there is a constant $C_{1}$ such that
\[
R+\left \vert \nabla \psi \right \vert ^{2}=\left(  4n-2\right)  \psi+C_{1};
\]
specifically, we have
\[
\left \vert \nabla \psi \right \vert \leq \sqrt{4n-2}\sqrt{\psi+C_{2}}%
\]
where $C_{2}=\frac{C_{1}}{4n-2}$.
\end{lemma}

\begin{proof}
Let $\left \{  e_{j}\right \}  _{j=1}^{2n}$ be a local orthonormal frame in
$\Gamma \left(  D\right)  $ and $e_{0}=\xi$. The Sasaki-Ricci soliton reads
\[
R_{jk}^{T}-\left(  2n+2\right)  g_{jk}=-\psi_{jk}.
\]
From the equality $Ric^{T}=Ric+2g$ on $\Gamma \left(  D\right)  $, we see that
\[
R_{jk}=2ng_{jk}-\psi_{jk}%
\]
and
\[
R=R_{jj}+R_{00}=4n^{2}-\Delta_{B}\psi+2n.
\]
By the calculation%
\[%
\begin{array}
[c]{ccl}%
\frac{1}{2}R_{,i} & = & R_{,i}-R_{ij,j}\\
& = & -\psi_{,jji}+\psi_{,jij}\\
& = & R_{ljji}\psi_{,l}\\
& = & \left(  R_{li}-R_{l00i}\right)  \psi_{,l}\\
& = & R_{ij}\psi_{,j}-\psi_{,i}%
\end{array}
,
\]
we have
\[
R_{,i}=2R_{ij}\psi_{,j}-2\psi_{,i}.
\]
Hence the equality
\[
\left(  R+\left \vert \nabla \psi \right \vert ^{2}-\left(  4n-2\right)
\psi \right)  _{,i}=0
\]
holds. With the help of this and that the scalar curvature $R$ and Hamiltonian
function $\psi$ are both basic, there is a constant $C_{1}$ such that
\[
R+\left \vert \nabla \psi \right \vert ^{2}-\left(  4n-2\right)  \psi=C_{1}.
\]
From the hypothesis of nonnegativity of the scalar curvature, we get
\[
\left \vert \nabla \psi \right \vert \leq \sqrt{4n-2}\sqrt{\psi+C_{2}}.
\]

\end{proof}

\begin{proposition}
\label{prop1}Under the same assumptions as precedes, we have the potential
estimate%
\[
n\left(  d\left(  x,y\right)  -7\right)  _{+}^{2}\leq \psi \left(  x\right)
+C_{2}\leq n\left(  d\left(  x,y\right)  +\sqrt{3}\right)  ^{2}%
\]
where $y$ is a minimum point of $\psi$.
\end{proposition}

\begin{proof}
Let $\gamma:\left[  0,s_{0}\right]  $ $\longrightarrow M$ be the minimal
normal geodesic from $x=\gamma \left(  0\right)  $ to $y=\gamma \left(
s_{0}\right)  $, may assume $s_{0}>2$ and
\[
\phi \left(  t\right)  =\left \{
\begin{array}
[c]{cl}%
t & on\text{ }\left[  0,1\right] \\
1 & on\text{ }\left[  1,s_{0}-1\right] \\
s_{0}-t & on\text{ }\left[  s_{0}-1,s_{0}\right]
\end{array}
\right.  .
\]
By the second variation of arc length, we have%
\[
\int_{0}^{s_{0}}\phi^{2}Ric\left(  \gamma^{\prime},\gamma^{\prime}\right)
\leq2n\int_{0}^{s_{0}}\left(  \phi^{\prime}\right)  ^{2}=4n.
\]
Because
\[%
\begin{array}
[c]{ccl}%
Ric\left(  \gamma^{\prime},\gamma^{\prime}\right)  & = & Ric\left(  \gamma
_{D}^{\prime},\gamma_{D}^{\prime}\right)  +2n\left(  \eta \left(
\gamma^{\prime}\right)  \right)  ^{2}\\
& = & Ric^{T}\left(  \gamma_{D}^{\prime},\gamma_{D}^{\prime}\right)
-2g\left(  \gamma_{D}^{\prime},\gamma_{D}^{\prime}\right)  +2n\left(
\eta \left(  \gamma^{\prime}\right)  \right)  ^{2}\\
& = & 2ng\left(  \gamma_{D}^{\prime},\gamma_{D}^{\prime}\right)  +2n\left(
\eta \left(  \gamma^{\prime}\right)  \right)  ^{2}-\nabla_{\gamma^{\prime}%
}\nabla_{\gamma^{\prime}}\psi \\
& = & 2n-\nabla_{\gamma^{\prime}}\nabla_{\gamma^{\prime}}\psi
\end{array}
\]
where $\gamma^{\prime}=\gamma_{D}^{\prime}+\eta \left(  \gamma^{\prime}\right)
\xi$, it is easy to find that%
\[
2n\left(  s_{0}-\frac{4}{3}\right)  =2n\int_{0}^{s_{0}}\phi^{2}\leq4n+\int
_{0}^{s_{0}}\phi^{2}\nabla_{\gamma^{\prime}}\nabla_{\gamma^{\prime}}\psi.
\]
Therefore, the lower bound of the potential $\psi$ could be deduced as below:%
\[%
\begin{array}
[c]{cl}
& 2nd\left(  x,y\right)  -\frac{20}{3}n\\
\leq & \int_{0}^{s_{0}}\phi^{2}\nabla_{\gamma^{\prime}}\nabla_{\gamma^{\prime
}}\psi \\
= & 2\int_{s_{0}-1}^{s_{0}}\phi \nabla_{\gamma^{\prime}}\psi-2\int_{0}^{1}%
\phi \nabla_{\gamma^{\prime}}\psi \\
\leq & \sqrt{4n-2}\left(  \sqrt{\psi \left(  x\right)  +C_{2}}+\sqrt
{\psi \left(  y\right)  +C_{2}}\right)  +\sqrt{4n-2}%
\end{array}
\]
and
\begin{equation}
0\leq R\left(  y\right)  =\left(  4n-2\right)  \psi \left(  y\right)
+C_{1}\leq4n^{2}+2n \label{p1}%
\end{equation}
imply that
\[
n\left(  d\left(  x,y\right)  -7\right)  _{+}^{2}\leq \psi \left(  x\right)
+C_{2}.
\]
Here we employ the Lipschitz property of $\sqrt{\psi+C_{2}}$ and the equality
$\left(  \mathrm{ii}\right)  $ of (\ref{n1}) stated in the subsequent section.
As for the upper bound, it is derived from the fact that the function
$\sqrt{\psi+C_{2}}$ is $\sqrt{n-\frac{1}{2}}$-Lipschitz. More precisely,
combining such fact with the inequality (\ref{p1}), we obtain%
\[
\sqrt{\psi \left(  x\right)  +C_{2}}\leq \sqrt{n-\frac{1}{2}}d\left(
x,y\right)  +\sqrt{\psi \left(  y\right)  +C_{2}}\leq \sqrt{n-\frac{1}{2}%
}d\left(  x,y\right)  +\sqrt{n\frac{2n+1}{2n-1}};
\]
then, it concludes that
\[
\psi \left(  x\right)  +C_{2}\leq n\left(  d\left(  x,y\right)  +\sqrt
{3}\right)  ^{2}.
\]

\end{proof}

\section{Positivity of the Scalar Curvature for Sasaki-Ricci Solitons}

For convenience, we collect here several properties of Sasaki-Ricci solitons
derived in the preceding section. Let $\left(  M^{2n+1},g,X\right)  $ be a
complete Sasaki-Ricci soliton and the corresponding Hamiltonian function
(potential) $\psi$ with respect to $X$, then, for any fixed local orthonormal
frame $\left \{  e_{j}\right \}  _{j=1}^{2n}$ in $\Gamma \left(  D\right)  $. We
have%
\begin{equation}
\left \{
\begin{array}
[c]{cl}%
\left(  \mathrm{i}\right)  & R_{jk}=2ng_{jk}-\psi_{,jk}\\
\left(  \mathrm{ii}\right)  & R+\Delta_{B}\psi=4n^{2}+2n\\
\left(  \mathrm{iii}\right)  & R_{,i}=2R_{ij}\psi_{,j}-2\psi_{,i}\\
\left(  \mathrm{iv}\right)  & R+\left \vert \nabla \psi \right \vert ^{2}=\left(
4n-2\right)  \psi+C_{1}\\
\left(  \mathrm{v}\right)  & \Delta_{B}R+2\left \vert Ric_{D}\right \vert
^{2}-2\left(  2n+1\right)  R+4n\left(  4n+1\right)  =g\left(  \nabla
R,\nabla \psi \right)  .
\end{array}
\right.  \label{n1}%
\end{equation}
Here $\left \vert Ric_{D}\right \vert ^{2}\triangleq%
{\displaystyle \sum \limits_{i,j=1}^{2n}}
\left(  R_{ij}\right)  ^{2}=$ $\left \vert Ric^{T}\right \vert ^{2}+8n-4R^{T}$
denotes the squared norm of the Ricci curvature restricted to the horizontal
distribution $\Gamma \left(  D\right)  $ and $C_{1}$ is a constant. Besides the
last equality $\left(  \mathrm{v}\right)  $ in (\ref{n1}), the rest ones has
been proved previously. Its derivation is similar to the Riemannian case;
however, for the sake of completeness, we sketch the proof as below.

\begin{proposition}
The equality $\left(  \mathrm{v}\right)  $ in (\ref{n1}) holds.
\end{proposition}

\begin{proof}
By the commutation formula and the first equation $\left(  \mathrm{i}\right)
$ in (\ref{n1}), we have
\[%
\begin{array}
[c]{ccl}%
-\left(  R_{jklm}\psi_{,m}\right)  _{,i} & = & -\left(  R_{kl}-2ng_{kl}%
\right)  _{,ji}+\left(  R_{jl}-2ng_{jl}\right)  _{,ki}\\
& = & -R_{kl,ji}+R_{jl,ki}%
\end{array}
.
\]
Hereafter the indices are with respect to the local orthonormal frame
$\left \{  e_{j}\right \}  _{j=1}^{2n}$. Taking the trace of both sides with
respect to $i$ and $j$, it becomes%
\[
\Delta_{B}\left(  R_{kl}\right)  -\frac{1}{2}R_{,kl}-R_{sk}R_{sl}%
+R_{kl}+2R_{jklm}R_{jm}-2nR_{jklj}=R_{jklm,j}\psi_{,m}%
\]
where we utilized the commutation formula, the first equation $\left(
\mathrm{i}\right)  $ in (\ref{n1}) and
\[
R_{jskj}R_{sl}=R_{sk}R_{sl}-R_{kl}.
\]
Subsequently, by
\[
R_{kjjl}=R_{kl}-g_{kl}%
\]
and the second Bianchi identity, we obtain%
\[
\Delta_{B}\left(  R_{kl}\right)  -\frac{1}{2}R_{,kl}-R_{sk}R_{sl}-\left(
2n-1\right)  R_{kl}+2ng_{kl}+2R_{jklm}R_{jm}=\left(  -R_{km,l}+R_{kl,m}%
\right)  \psi_{,m}.
\]
Then, taking the trace again with respect to $k$ and $l$, it could be deduced
that
\[
\Delta_{B}R+2\left \vert Ric_{D}\right \vert ^{2}-2\left(  2n+1\right)
R+4n\left(  4n+1\right)  =g\left(  \nabla R,\nabla \psi \right)  .
\]

\end{proof}

\bigskip Note that the last equality could be also expressed in%
\[
\Delta_{B,\psi}R^{T}+2\left \vert Ric^{T}\right \vert ^{2}-2\left(  2n+5\right)
R^{T}+24n\left(  n+1\right)  =0.
\]

Now we are able to show the nonnegativity of the transverse scalar curvature
of Sasaki-Ricci solitons by following the elliptic method adopted by Zhu-Hong
Zhang \cite{zha09.2} (or see \cite{chowetal15}). Another way towards the fact
by the parabolic method could be referred to the original paper \cite{che09}
of B.-L. Chen.

\begin{theorem}
\label{nt}The scalar curvature of any complete noncompact Sasaki-Ricci soliton
$\left(  M^{2n+1},g,X\right)  $ is positive.
\end{theorem}

\begin{proof}
Fix a point $p$ in $M$. Set $r\left(  x\right)  =d\left(  p,x\right)  $ and
$\eta:\left[  0,+\infty \right)  \longrightarrow \left[  0,1\right]  $ to be the
distance function and the cut-off function respectively. The cut-off function
$\eta$ is defined to be a smooth decreasing function which satisfies
\[
\eta \left(  t\right)  =\left \{
\begin{array}
[c]{cl}%
1 & \mathrm{on}\text{ }\left[  0,1\right] \\
0 & \mathrm{on}\text{ }\left[  3,+\infty \right)
\end{array}
\right.
\]
and $\eta^{\prime \prime}-2\frac{\left(  \eta^{\prime}\right)  ^{2}}{\eta}$ is
bounded from below by a uniform constant, say $-C_{0}$. Let $\phi_{c}\left(
x\right)  \triangleq \eta \left(  \frac{r\left(  x\right)  }{c}\right)  R\left(
x\right)  $ where $r\left(  x\right)  =d\left(  p,x\right)  $ and $c\geq2$.
Hereafter, we drop the notation "$\circ \frac{r\left(  x\right)  }{c}$" in the
expression involving the function $\eta$. First, we compute the basic weighted
$\psi$-Laplacian of $\phi_{c}$:%
\begin{equation}%
\begin{array}
[c]{ccl}%
\Delta_{B,\psi}\phi_{c} & = & \eta \Delta_{B,\psi}R+\frac{2}{c}\eta^{\prime
}g\left(  \nabla r,\nabla R\right)  +\left(  \frac{\eta^{\prime}}{c}%
\Delta_{B,\psi}r+\frac{\eta^{\prime \prime}}{c^{2}}\left \vert \nabla
r\right \vert ^{2}\right)  R\\
& = & \eta \left[  -2\left \vert Ric_{D}\right \vert ^{2}+2\left(  2n+1\right)
R-4n\left(  4n+1\right)  \right]  +\\
&  & \frac{2}{c}\frac{\eta^{\prime}}{\eta}g\left(  \nabla r,\nabla \phi
_{c}\right)  +\frac{\eta^{\prime}}{c}R\Delta_{B,\psi}r+\frac{1}{c^{2}}\left(
\eta^{\prime \prime}-2\frac{\left(  \eta^{\prime}\right)  ^{2}}{\eta}\right)  R
\end{array}
\label{n2}%
\end{equation}
on $\left[  \eta \neq0\right]  $ with the help of the equality $\left(
\mathrm{v}\right)  $ in (\ref{n1}). Assume $x_{c}$ achieves the minimum of
$\phi_{c}$, i.e. $\phi_{c}\left(  x_{c}\right)  =\underset{M}{\text{ }\min
}\phi_{c}<0$. So $R\left(  x_{c}\right)  <0$. Then, by applying the maximum
principle to (\ref{n2}) and the fact that%
\[
R^{2}\leq2n\left \vert Ric_{D}\right \vert ^{2}+4nR+12n^{2},
\]
we have, at the minimum point $x_{c}$,
\begin{equation}
0\geq \eta \left[  2\left(  2n+3\right)  -\frac{R}{n}+8n\frac{\left(
1-2n\right)  }{R}\right]  +\frac{\eta^{\prime}}{c}\Delta_{B,\psi}r+\frac
{1}{c^{2}}\left(  \eta^{\prime \prime}-2\frac{\left(  \eta^{\prime}\right)
^{2}}{\eta}\right)  . \label{n3}%
\end{equation}
When $x_{c}\in B\left(  p,c\right)  $, it could be observed that
\[%
\begin{array}
[c]{ccl}%
0 & \geq & 2\left(  2n+3\right)  -\frac{R\left(  x_{c}\right)  }{n}\\
& = & 2\left(  2n+3\right)  -\frac{\phi_{c}\left(  x_{c}\right)  }{n}\\
& \geq & 2\left(  2n+3\right)  -\frac{\phi_{c}\left(  x\right)  }{n}\\
& = & 2\left(  2n+3\right)  -\frac{1}{n}R\left(  x\right)
\end{array}
\]
on $B\left(  p,c\right)  $. This implies that $R\geq2n\left(  2n+3\right)  >0$
on $B\left(  p,c\right)  $. This contradicts with $R\left(  x_{c}\right)  <0$.
Accordingly, the point $x_{c}$ lies in the annulus $B\left(  p,3c\right)
\backslash B\left(  p,c\right)  $. With the help of the equality $\left(
\mathrm{i}\right)  $ in (\ref{n1}), Lemma 18.6 in \cite{chowetal10}, and the
equality%
\[%
{\displaystyle \int \limits_{\left[  0,r\left(  x_{c}\right)  \right]  }}
Hess\left(  \psi \right)  \left(  \gamma^{\prime},\gamma^{\prime}\right)
=g\left(  \left(  \nabla \psi \right)  \left(  x_{c}\right)  ,\gamma^{\prime
}\left(  r\left(  x_{c}\right)  \right)  \right)  -g\left(  \left(  \nabla
\psi \right)  \left(  p\right)  ,\gamma^{\prime}\left(  0\right)  \right)
\]
where $\gamma:$ $\left[  0,r\left(  x_{c}\right)  \right]  \longrightarrow M$
denotes the normal minimal geodesic from $p$ to $x_{c}$, we see that
\[%
\begin{array}
[c]{ccl}%
\Delta_{B,\psi}r & \leq & \left(  n-1\right)
{\displaystyle \int \limits_{\left[  0,r\left(  x_{c}\right)  \right]  }}
\left(  \zeta^{\prime}\right)  ^{2}-g\left(  \left(  \nabla \psi \right)
\left(  p\right)  ,\gamma^{\prime}\left(  0\right)  \right) \\
&  & -2nr\left(  x_{c}\right)  +%
{\displaystyle \int \limits_{\left[  0,r\left(  x_{c}\right)  \right]  }}
\left(  1-\zeta^{2}\right)  Ric\left(  \gamma^{\prime},\gamma^{\prime}\right)
\end{array}
\]
for any continuous, piecewise smooth function $\zeta:\left[  0,r\left(
x_{c}\right)  \right]  \longrightarrow M$ with $\zeta \left(  0\right)  =0$ and
$\zeta \left(  r\left(  x_{c}\right)  \right)  =1$. Take the choice
\[
\zeta \left(  s\right)  =\left \{
\begin{array}
[c]{cl}%
s & \mathrm{on}\text{ }\left[  0,1\right] \\
1 & \mathrm{on}\text{ }\left[  1,r\left(  x_{c}\right)  \right]
\end{array}
\right.  .
\]
The last inequality becomes
\begin{equation}
\Delta_{B,\psi}r\leq \left(  n-1\right)  -g\left(  \left(  \nabla \psi \right)
\left(  p\right)  ,\gamma^{\prime}\left(  0\right)  \right)  -2nr\left(
x_{c}\right)  +\frac{2}{3}\max_{\overline{B\left(  p,1\right)  }}Ric^{+}.
\label{n4}%
\end{equation}
Here $Ric^{+}\triangleq \max \left \{  Ric,0\right \}  $ and%
\[
\max_{\overline{B\left(  p,1\right)  }}Ric^{+}\triangleq \max_{\substack{x\in
\overline{B\left(  p,1\right)  }\\v_{x}\in T_{x}M\cap S^{2n}}}Ric^{+}\left(
x\right)  \left(  v_{x},v_{x}\right)  .
\]
Combing the inequalities (\ref{n3}) and (\ref{n4}), we are able to get, on
$B\left(  p,c\right)  $
\[%
\begin{array}
[c]{cl}
& \frac{1}{n}R\left(  x\right) \\
\geq & \frac{1}{n}\phi_{c}\left(  x_{c}\right) \\
\geq & \frac{\eta^{\prime}}{c}\left[  \left(  n-1\right)  -g\left(  \left(
\nabla \psi \right)  \left(  p\right)  ,\gamma^{\prime}\left(  0\right)
\right)  +\frac{2}{3}\underset{\overline{B\left(  p,1\right)  }}{\max}%
Ric^{+}\right] \\
& -\frac{2n\eta^{\prime}}{c}r\left(  x_{c}\right)  +2\left(  2n+3\right)
\eta-\frac{C_{1}}{c^{2}}\\
\geq & -\frac{C_{2}}{c}\left[  \left(  n-1\right)  +\left \vert \left(
\nabla \psi \right)  \right \vert \left(  p\right)  +\frac{2}{3}\underset
{\overline{B\left(  p,1\right)  }}{\max}Ric^{+}\right]  -\frac{C_{3}}{c^{2}}%
\end{array}
\]
for the uniform constants $C_{1}$,$C_{2}$,$C_{3}$. Therefore, as
$c\longrightarrow+\infty$, the scalar curvature of any complete noncompact
Sasaki-Ricci soliton $\left(  M^{2n+1},g,X\right)  $ is nonnegative. By
applying the maximum principle to the equation $\left(  \mathrm{v}\right)  $
in (\ref{n1}), we find that the scalar curvature is actually positive.
\end{proof}

\section{A Criterion for the Transverse Rigidity of Sasaki-Ricci Solitons}

In this section, we establish a criterion which characterize when a
Sasaki-Ricci soliton is transversely rigid, and subsequently it enable us to
classify the Sasaki-Ricci solitons of constant scalar curvature at least in
low-dimensional cases. Before presenting the criterion, we first show the
quantization of the scalar curvature of a Sasaki-Ricci soliton with constant
scalar curvature by studying geometric data which envolved from the
isoparametric theory.

\begin{proposition}
\label{Pa1}(= Proposition 1.)If $\left(  M^{2n+1},g,X\right)  $ is a complete
Sasaki-Ricci soliton with constant scalar curvature $R$ and the Hamiltonian
potential $\psi$ with respect to $X$, then $R$ must lie in the finite set
$\left \{  \left(  2n-1\right)  k+\left(  2n+1\right)  \right \}  _{k=1}^{2n+1}$.
\end{proposition}

\begin{proof}
From Proposition \ref{prop1}, we see that the focal variety $M_{\min}%
\doteq \left[  \psi=\underset{M}{\text{ }\min}\psi \right]  $ is nonempty.
Qi-Ming Wang had showed that such variety $M_{\min}$ is actually a smooth
minimal submanifold and the Hessian operator restricted to the focal
submanifold $M_{\min}$ has only two distinct eigenvalues up to the
multiplicity
\[
Hess\left(  \psi \right)  |_{M_{\min}}=\left \{
\begin{array}
[c]{cc}%
0 & \mathrm{on}\text{ }\Gamma \left(  TM_{\min}\otimes TM_{\min}\right) \\
\left(  2n-1\right)  g & \mathrm{on}\text{ }\Gamma \left(  TM_{\min}^{\perp
}\otimes TM_{\min}^{\perp}\right)
\end{array}
\right.
\]
by adapting his conclusion to our case, that is the isoparametric function
determined by the equation $\left(  \mathrm{ii}\right)  $ as well as the one
$\left(  \mathrm{iv}\right)  $ of (\ref{n1}). And then, by the equation of a
Sasaki-Ricci soltion, we get
\begin{equation}
Ric|_{M_{\min}}=\left(
\begin{array}
[c]{cc}%
2nI_{k} & 0\\
0 & I_{\left(  2n+1\right)  -k}%
\end{array}
\right)  _{\left(  2n+1\right)  \times \left(  2n+1\right)  }. \label{a1}%
\end{equation}
Here $I_{k}$ denotes the identity $k\times k$-matrix and $k=\dim M_{\min}$. It
is worth to note that the Reeb vector field is tangent to the focal
submanifold $M_{\min}$ so that $k\geq1$. From the expression (\ref{a1}) and
the assumption of constant scalar curvature, it could be observed that the
possible values of the scalar curvature are quantified and lie in the finite
set $\left \{  \left(  2n-1\right)  k+\left(  2n+1\right)  \right \}
_{k=1}^{2n+1}$.
\end{proof}

Now the main result of this paper could be manifested as below.

\begin{theorem}
\label{criteria1}(= Theorem 1.)If $\left(  M,g,X\right)  $ is a complete
Sasaki-Ricci soliton with the Hamiltonian potential $\psi$ with respect to
$X$, then the following are equivalent: \newline(i) $\left(  M,g,X\right)  $
is transversely rigid. \newline(ii) The scalar curvature is constant and it is
transversely radially flat, i.e. \newline$R\left(  \cdot,\nabla^{T}%
\psi \right)  \nabla^{T}\psi$ vanishes along the horizontal direction.
\newline(iii) The scalar curvature is constant and $g\leq Ric\leq2ng$.
\end{theorem}

\begin{proof}
It is not difficult to observe that the implication (i)$\Rightarrow$(iii).
Once we verify the deduction (iii)$\Rightarrow$(ii)$\Rightarrow$(i), the proof
is accomplished. First, we show the case of (iii)$\Rightarrow$(ii). From the
equations $\left(  \mathrm{i}\right)  $ and $\left(  \mathrm{iii}\right)  $ of
(\ref{n1}), it ensures that the coefficients of $\nabla_{e_{k},\nabla^{T}\psi
}^{2}\nabla^{T}\psi$ could be expressed as%
\[%
\begin{array}
[c]{cll}%
\left(  \left(  \nabla^{T}\right)  _{e_{k},\nabla^{T}\psi}^{2}\nabla^{T}%
\psi \right)  ^{i} & = & \psi_{ij,k}\psi_{j}=-R_{ij,k}\psi_{j}\\
& = & -\left[  \left(  R_{ij}\psi_{j}\right)  _{,k}-R_{ij}\psi_{jk}\right] \\
& = & -\left[  \left(  \frac{1}{2}R_{,i}+\psi_{i}\right)  _{,k}-R_{ij}\left(
2ng_{jk}-R_{jk}\right)  \right] \\
& = & -\frac{1}{2}R_{,ik}-\left(  2ng_{ik}-R_{ik}\right)  +R_{ij}\left(
2ng_{jk}-R_{jk}\right) \\
& = & -\frac{1}{2}R_{,ik}+\left(  R_{ij}-g_{ij}\right)  \left(  2ng_{jk}%
-R_{jk}\right)  .
\end{array}
\]
So we have
\begin{equation}
R\left(  e_{k},\nabla^{T}\psi,\nabla^{T}\psi,e_{i}\right)  =-\frac{1}%
{2}R_{,ik}+\left(  R_{ij}-g_{ij}\right)  \left(  2ng_{jk}-R_{jk}\right)
+R_{ik,j}\psi_{j} \label{a4}%
\end{equation}
and then, by taking the trace, it is obtained that
\begin{equation}%
\begin{array}
[c]{cl}
& \frac{1}{2}R_{,i}\psi_{i}\\
= & Ric\left(  \nabla^{T}\psi,\nabla^{T}\psi \right)  -\left \vert \nabla
^{T}\psi \right \vert ^{2}\\
= &
{\displaystyle \sum \limits_{k}}
R\left(  e_{k},\nabla^{T}\psi,\nabla^{T}\psi,e_{k}\right) \\
= & -\frac{1}{2}\Delta_{B}R+\left(  R_{kj}-g_{kj}\right)  \left(
2ng_{jk}-R_{jk}\right)  +R_{,j}\psi_{j}%
\end{array}
\label{a2}%
\end{equation}
with the help of the equation $\left(  \mathrm{iii}\right)  $ of (\ref{n1})
and the identity
\[
R\left(  \cdot,\cdot \right)  \xi=id\wedge \eta.
\]
Therefore, the equation (\ref{a2}) leads to%
\begin{equation}
\frac{1}{2}\Delta_{B,\psi}R=tr\left(  \left(  Ric_{D}-g^{T}\right)  \left(
2ng^{T}-Ric_{D}\right)  \right)  . \label{a3}%
\end{equation}
For if the conditions of (iii) hold, then (\ref{a3}) suggests the restriction
$Ric_{D}$ of Ricci curvature to the contact bundle $D$ are either $g^{T}$ or
$2ng^{T}$. Thus, combining with such result, the equation (\ref{a4}) becomes
\begin{equation}
R\left(  \cdot,\nabla^{T}\psi \right)  \nabla^{T}\psi=-\left(  \nabla
^{T}\right)  _{\nabla^{T}\psi,\cdot}^{2}\nabla^{T}\psi. \label{a5}%
\end{equation}
In the following, we would claim
\[
R\left(  \cdot,\nabla^{T}\psi,\nabla^{T}\psi,\cdot \right)  =0.
\]
This could be achieved by proving $R\left(  \cdot,\nabla^{T}\psi,\nabla
^{T}\psi,\cdot \right)  $ is symmetric and anti-symmetric. It is clear that
$R\left(  \cdot,\nabla^{T}\psi,\nabla^{T}\psi,\cdot \right)  $ is symmetric.
Hence, it suffices to tackle whether it is anti-symmetric. Denote by $\left \{
E_{j}\right \}  _{j}$ an orthonormal basis of subHessian operator $Hess_{b}%
\psi$ on the contact bundle consists of eigenvectors corresponding to the
eigenvalues $\lambda_{j}$, i.e. $\nabla_{E_{j}}\nabla^{T}\psi=\lambda_{j}%
E_{j}$. Apparently, these eigenvalues lie in the set $\left \{  0,2n-1\right \}
$. By (\ref{a5}), it is easily observed that for the indices $j,k$%
\[%
\begin{array}
[c]{cl}
& -R\left(  E_{j},\nabla^{T}\psi,\nabla^{T}\psi,E_{k}\right) \\
= & g\left(  \left(  \nabla^{T}\right)  _{\nabla^{T}\psi,E_{j}}^{2}\nabla
^{T}\psi,E_{k}\right) \\
= & g\left(  \nabla_{\nabla^{T}\psi}^{T}\nabla_{E_{j}}^{T}\nabla^{T}\psi
,E_{k}\right)  -g\left(  \nabla_{\nabla_{\nabla^{T}\psi}^{T}E_{j}}^{T}%
\nabla^{T}\psi,E_{k}\right) \\
= & \lambda_{j}g\left(  \nabla_{\nabla^{T}\psi}^{T}E_{j},E_{k}\right)
-g\left(  \nabla_{E_{k}}^{T}\nabla^{T}\psi,\nabla_{\nabla^{T}\psi}^{T}%
E_{j}\right) \\
= & g\left(  \nabla_{\nabla^{T}\psi}^{T}E_{j},E_{k}\right)  \left(
\lambda_{j}-\lambda_{k}\right) \\
= & -g\left(  \nabla_{\nabla^{T}\psi}^{T}E_{k},E_{j}\right)  \left(
\lambda_{j}-\lambda_{k}\right) \\
= & R\left(  E_{k},\nabla^{T}\psi,\nabla^{T}\psi,E_{j}\right)
\end{array}
\]
where we utilize the compatibility of $\nabla^{T}$ with $g^{T}$ in the fifth
equality and the fact that the quantity $\left(  \lambda_{j}-\lambda
_{k}\right)  $ is just a constant. Therefore, the part of (iii)$\Rightarrow
$(ii) is completed. As for the case of of (ii)$\Rightarrow$(i), in the proof
of Proposition \ref{Pa1}, we learned that the focal variety $M_{\min}$ is
actually a compact totally geodesic Einstein manifold. Then, by following the
idea of Petersen-Wylie, it could be found that such Sasaki-Ricci soliton is
transversely rigid. Consequently, the equivalence of these three conditions is established.
\end{proof}

\begin{remark}
It is worth to note that the above deduction about the transversely radial
flatness is a little bit different from the one in the course of proof "(iii)
implies (ii)". More precisely, the method adopted by Petersen-Wylie heavily
depends on the fact that Hessian operator of the potential admits the only two
possible eigenvalues $0$ and $\lambda$.
\end{remark}

Combining with Theorem \ref{nt}, it is apparent to derive the following
characterization about the transverse rigidity of Sasaki-Ricci solitons. More
precisely, by the $D$-homothetic transformation, it could be observed that the
existence of a Sasaki-Einstein metric is equivalent to the existence of a
Sasaki $\eta$-Einstein metric of positive transverse scalar curvature (see
\cite{fow09}).

\begin{corollary}
If $\left(  M,g,X\right)  $ is a complete Sasaki-Ricci soliton, then it is
transversely rigid if and only if it is Sasaki-Einstein.
\end{corollary}

By the aforementioned result, we are able to derive another type of criterion
for determining when the transverse rigidity holds under the hypothesis of
constant scalar curvature as follows.

\begin{theorem}
\label{criteria2}(= Theorem 2.)A complete Sasaki-Ricci soliton $\left(
M,g,X\right)  $ of constant scalar curvature is Sasaki-Einstein if and only if
the rank of the operator $Ric-g$ is constant.
\end{theorem}

\begin{proof}
In virtue of the explicit implication for the necessity, it suffices to
clarify if the operator $Ric-g$ admits constant rank and it is of constant
scalar curvature, then the transverse rigidity holds. Following the notations
in the proof of Proposition \ref{Pa1}, the assumption of the constant rank of
$Ric-g$ gives us the rank is actually $k=\dim M_{\min}$. Hence, we denote by
$R_{1}=2n,R_{2},\cdot \cdot \cdot,R_{k},R_{j}=1$ $\mathrm{for}$ $k+1\leq
j\leq2n+1$, the eigenvalues of the Ricci operator on $M$, and then we get
\begin{equation}
R-\left(  2n+1\right)  =tr\left(  Ric-g\right)  =%
{\displaystyle \sum \limits_{j=1}^{k}}
\left(  R_{j}-1\right)  =\left(
{\displaystyle \sum \limits_{j=1}^{k}}
R_{j}\right)  -k. \label{a6}%
\end{equation}
Combining with the equation $\left(  \mathrm{v}\right)  $ of (\ref{n1}) which
provides the equality
\begin{equation}%
\begin{array}
[c]{cl}
& \left(  2n+1\right)  R-2n\left(  4n+1\right)  =\left \vert Ric_{D}\right \vert
^{2}\\
= &
{\displaystyle \sum \limits_{j=2}^{2n+1}}
R_{j}^{2}=\left(
{\displaystyle \sum \limits_{j=1}^{k}}
R_{j}^{2}\right)  +\left(  2n-k+1\right)  -4n^{2},
\end{array}
\label{a7}%
\end{equation}
these enables us to observe $%
{\displaystyle \sum \limits_{j=1}^{k}}
\left(  R_{j}-2n\right)  ^{2}=0$. So it furnishes that the eigenvalues of the
Ricci operator lie in the interval $\left[  1,2n\right]  $. The theorem is
concluded by applying the criterion as precedes.
\end{proof}

As an application, we shall show the transverse rigidity and the nonexistence
of the Sasaki-Ricci soliton whose constant scalar curvature attains the three
extremal values as follows.

\begin{proposition}
A complete Sasaki-Ricci soliton $\left(  M,g,X\right)  $ with constant scalar
curvature is Sasaki-Einstein when the scalar curvature equals either
$4n^{2}+1$ or $4n^{2}+2n$, i.e. $k=2n,2n+1$. Furthermore, if the scalar
curvature is equal to $4n$, i.e. $k=1$, then there is no complete Sasaki-Ricci soliton.
\end{proposition}

\begin{proof}
With the same notations as above and by the identity (\ref{a3}), we have%
\begin{equation}%
\begin{array}
[c]{cl}
& \frac{1}{2}\Delta_{B,\psi}R\\
= & tr\left(  \left(  Ric_{D}-g^{T}\right)  \left(  2ng^{T}-Ric_{D}\right)
\right) \\
= & -\left \vert Ric_{D}\right \vert ^{2}+\left(  2n+1\right)  R_{D}-4n^{2}\\
= & -\left[  \left \vert Ric_{D}-\frac{R_{D}}{2n}g^{T}\right \vert ^{2}%
+\frac{R_{D}^{2}}{2n}\right]  +\left(  2n+1\right)  R_{D}-4n^{2}\\
= & -\left \vert Ric_{D}-\frac{R_{D}}{2n}g^{T}\right \vert ^{2}+\frac{1}%
{2n}\left(  R-4n\right)  \left[  2n\left(  2n+1\right)  -R\right]
\end{array}
\label{a8}%
\end{equation}
where the equation $R=R_{D}+2n$ is used in the last equality. From this, if
the scalar curvature $R$ equals $4n^{2}+2n$, it could be observed that
\[
Ric^{T}=Ric_{D}+2g^{T}=2\left(  n+1\right)  g^{T}.
\]
So it is Sasaki-Einstein. For the case of the constant scalar curvature
$4n^{2}+1$, it could be computed%
\[%
\begin{array}
[c]{cl}
&
{\displaystyle \sum \limits_{j=1}^{2n}}
\left(  R_{j}-2n\right)  ^{2}\\
= & \left(
{\displaystyle \sum \limits_{j=1}^{2n}}
R_{j}^{2}\right)  -4n\left(  R-1\right)  +8n^{3}\\
= & \left[  \left(  2n+1\right)  R-2n\left(  4n+1\right)  -1+4n^{2}\right]
-4n\left(  R-1\right)  +8n^{3}\\
= & 0
\end{array}
\]
by the equalities (\ref{a6}) and (\ref{a7}). The transverse rigidity is
implied by Theorem \ref{criteria2} (or by Theorem \ref{criteria1}). On the
other hand, when the scalar curvature $R$ is equal to $4n$, with the same
reason in Manuel Fern\'{a}ndez-L\'{o}pez and Eduardo Garc\'{\i}a-R\'{\i}o's
paper, we also reach a contradiction with the finite fundamental group of
$\left(  M,g,X\right)  $. Accordingly, in this case, there does not exist any
complete Sasaki-Ricci soliton.
\end{proof}

\begin{remark}
By the equation (\ref{a8}), when the scalar curvature $R$ equals $4n$, it
could be obtained that%
\[
Ric^{T}=3g^{T}.
\]
In some sense, this provides an evidence of nonexistence of Sasaki-Ricci
solitons in this case.
\end{remark}

In general, the complete Sasaki-Ricci soliton may not be transversely rigid;
nonetheless, we can verify the transverse rigidity of low-dimensional
Sasaki-Ricci soliton through analyzing the behavior of at most four distinct
eigenvalues of Ricci operator. As an application, we show that any complete
low-dimensional Sasaki-Ricci soliton with constant scalar curvature must be Sasaki-Einstein.

\begin{corollary}
(= Corollary 1.)If $\left(  M,g,X\right)  $ is a complete Sasaki-Ricci soliton
of constant scalar curvature with at most four distinct eigenvalues of Ricci
operator, then it is Sasaki-Einstein. As an application, any complete
Sasaki-Ricci soliton of dimension at most seven with constant scalar curvature
must be Sasaki-Einstein.
\end{corollary}

\begin{proof}
With the help of the preceding proposition, may assume the scalar curvature
$\left(  2n-1\right)  k+\left(  2n+1\right)  $ for $k\leq2n$. Then, by
observing that the vector fields $\xi$ and $\nabla^{T}\psi$ occupy the two
seats of dissimilar eigenvalues from the fact that
\[
Ric\left(  \cdot,\xi \right)  =2n\eta
\]
and $\left(  \mathrm{iii}\right)  $ in (\ref{n1}), we assume that there are
two distinct eigenvalues $R_{1}$ and $R_{2}$, different from $1$ and $2n$, of
the Ricci operator with corresponding multiplicities $k_{1}$ and $k_{2}$
without loss of generality. Following the same deduction of Manuel
Fern\'{a}ndez-L\'{o}pez and Eduardo Garc\'{\i}a-R\'{\i}o, it could be derived
that the eigenvalues $R_{1}$ and $R_{2}$ are both constant. This says that
such Sasaki-Ricci soliton is actually transverse rigid by Theorem
\ref{criteria2}. More precisely, from the equalities (\ref{a6}) and
(\ref{a7}), the eigenvalues $R_{1}$ and $R_{2}$ solve the system of equations%
\[
\left \{
\begin{array}
[c]{l}%
k_{1}R_{1}+k_{2}R_{2}=R-2n-1\\
k_{1}R_{1}^{2}+k_{2}R_{2}^{2}=\left(  2n+1\right)  R-8n^{2}-2n-1
\end{array}
.\right.
\]
Due to the continuity of the eigenvalues of the Ricci operator, it ensures the
multiplicities $k_{1}$ and $k_{2}$ are constant, and subsequently both
eigenvalues $R_{1}$ and $R_{2}$ are also constant. The rest shall be
immediately achieved by the fact the operator $Ric_{D}$ admits even
multiplicity which obtained from%
\[
Ric_{D}\left(  \phi,\phi \right)  =Ric_{D}.
\]

\end{proof}

\section{\bigskip Sasaki-Ricci Solitons with Harmonic Weyl Tensor}

The classification of complete Sasaki-Ricci solitons with harmonic Weyl tensor
will be proven in this section. We first present the following lemma, which
will play an important role in the proof of Proposition \ref{proposition2}.

\begin{lemma}
If $\left(  M,g,X\right)  $ is a complete Sasaki-Ricci soliton with the
Hamiltonian potential $\psi$ with respect to $X$, then%
\begin{equation}
\Delta_{B,\psi}R_{jl}=4n(R_{jl}-g_{jl})-2\sum_{p,q=1}^{2n}R_{pq}R_{jplq}.
\label{lem1}%
\end{equation}

\end{lemma}

\begin{proof}
From the second Bianchi identity and
\begin{equation}%
\begin{array}
[c]{cl}
& R_{i0j0}:=R(e_{i},\xi,e_{j},\xi)=-R(e_{i},\xi,\xi,e_{j})\\
= & -\langle R(e_{i},\xi)e_{j},\xi \rangle \\
= & -\langle g(\xi,e_{j})e_{i}-g(e_{i},e_{j})\xi,\xi \rangle \\
= & g(e_{i},e_{j})=\delta_{ij},
\end{array}
\label{SR3}%
\end{equation}
we derive
\begin{equation}%
\begin{array}
[c]{cl}
& \sum_{i=1}^{2n}R_{ijkl,i}=\sum_{i=1}^{2n}R_{ijil,k}-\sum_{i=1}%
^{2n}R_{ijik,l}\\
= & \sum_{\alpha=1}^{2n+1}R_{\alpha j\alpha l,k}-R_{0j0l,k}-\left(
\sum_{\alpha=1}^{2n+1}R_{\alpha j\alpha k,l}-R_{0j0k,l}\right) \\
= & (R_{jl}-g_{jl})_{,k}-(R_{jk}-g_{jk})_{,l}\\
= & R_{jl,k}-R_{jk,l}.
\end{array}
\label{lem2}%
\end{equation}
Then by
\begin{equation}
R_{ij}=2ng_{ij}-\psi_{ij}, \label{SR9}%
\end{equation}
\ref{lem2} and the Ricci identity, we get the equality
\begin{equation}
\sum_{i=1}^{2n}R_{ijkl,i}=R_{jl,k}-R_{jk,l}=\psi_{jk,l}-\psi_{jl,k}=\sum
_{i=1}^{2n}\psi_{i}R_{ijkl}. \label{lem3}%
\end{equation}
The Ricci identity implies%
\begin{equation}
R_{klpj,ip}=R_{klpj,pi}+\sum_{q=1}^{2n}(R_{qlpj}R_{qkip}+R_{kqpj}%
R_{qlip}+R_{klqj}R_{qpip}+R_{klpq}R_{qjip}). \label{lem10}%
\end{equation}
Similarly,%
\begin{equation}
R_{klpi,jp}=R_{klpi,pj}+\sum_{q=1}^{2n}(R_{qlpi}R_{qkjp}+R_{kqpi}%
R_{qljp}+R_{klqi}R_{qpjp}+R_{klpq}R_{qijp}). \label{lem11}%
\end{equation}
By (\ref{lem10}), (\ref{lem11}) and the second Bianchi identity, we obtain%
\begin{equation}%
\begin{array}
[c]{cl}
& \Delta_{B}R_{ijkl}=\sum_{p=1}^{2n}R_{ijkl,pp}\\
= & \sum_{p=1}^{2n}(R_{klpj,ip}-R_{klpi,jp})\\
= & \sum_{p=1}^{2n}(R_{klpj,pi}-R_{klpi,pj})\\
& +\sum_{p,q=1}^{2n}(R_{qlpj}R_{qkip}+R_{kqpj}R_{qlip}+R_{klqj}R_{qpip}%
+R_{klpq}R_{qjip})\\
& -\sum_{p,q=1}^{2n}(R_{qlpi}R_{qkjp}+R_{kqpi}R_{qljp}+R_{klqi}R_{qpjp}%
+R_{klpq}R_{qijp})\\
\triangleq & I_{1}+(I_{2}+I_{3}+I_{4}+I_{5})-(I_{6}+I_{7}+I_{8}+I_{9}).
\end{array}
\label{lem12}%
\end{equation}
From (\ref{lem3}), we derive
\begin{equation}%
\begin{array}
[c]{ccl}%
I_{1} & = & \sum_{p=1}^{2n}\left[  (R_{pjkl,p})_{,i}-(R_{klpi,p})_{,j}\right]
\\
& = & \sum_{p=1}^{2n}[(\psi_{p}R_{pjkl})_{,i}-(\psi_{p}R_{pikl})_{,j}]\\
& = & \sum_{p=1}^{2n}(R_{klpj,i}\psi_{p}-R_{klpi,j}\psi_{p}+R_{klpj}\psi
_{pi}-R_{klpi}\psi_{pj}).
\end{array}
\label{lem13}%
\end{equation}
By the identity (\ref{SR9}),
\begin{equation}%
\begin{array}
[c]{cl}
& R_{klpj}\psi_{pi}-R_{klpi}\psi_{pj}\\
= & R_{klpj}(2ng_{pi}-R_{pi})-R_{klpi}(2ng_{pj}-R_{pj})\\
= & 4nR_{klij}+R_{klpi}R_{pj}-R_{klpj}R_{pi}.
\end{array}
\label{lem14}%
\end{equation}
Using the second Bianchi identity again,%
\begin{equation}%
\begin{array}
[c]{cl}
& R_{klpj,i}\psi_{p}-R_{klpi,j}\psi_{p}\\
= & (R_{klpj,i}-R_{klpi,j})\psi_{p}\\
= & R_{klij,p}\psi_{p}=R_{ijkl,p}\psi_{p}.
\end{array}
\label{lem15}%
\end{equation}
Substituting (\ref{lem14}) and (\ref{lem15}) into (\ref{lem13}),%
\begin{equation}
I_{1}=\sum_{p=1}^{2n}R_{ijkl,p}\psi_{p}+4nR_{ijkl}+\sum_{p=1}^{2n}%
R_{klpi}R_{pj}-\sum_{p=1}^{2n}R_{klpj}R_{pi}. \label{lem16}%
\end{equation}
Direct computation gives us%
\begin{equation}
I_{2}-I_{7}=2\sum_{p,q=1}^{2n}R_{qlpj}R_{qkip}\text{ }\mathrm{and}\text{
}I_{3}-I_{6}=2\sum_{p,q=1}^{2n}R_{kqpj}R_{qlip}. \label{lem17}%
\end{equation}
The first Bianchi identity implies%
\begin{equation}
I_{5}-I_{9}=\sum_{p,q=1}^{2n}R_{klpq}(R_{qjip}-R_{qijp})=\sum_{p,q=1}%
^{2n}R_{klpq}R_{qpij}. \label{lem18}%
\end{equation}
By the equality (\ref{SR3}),%
\[
\sum_{p=1}^{2n}R_{qpip}=\sum_{\alpha=1}^{2n+1}R_{q\alpha i\alpha}%
-R_{q0i0}=R_{qi}-g_{qi}.
\]
It follows%
\begin{equation}
I_{4}-I_{8}=\sum_{q=1}^{2n}R_{klqj}(R_{qi}-g_{qi})-\sum_{q=1}^{2n}%
R_{klqi}(R_{qj}-g_{qj}). \label{lem20}%
\end{equation}
Substituting (\ref{lem16})-(\ref{lem20}) into (\ref{lem12}),%
\[%
\begin{array}
[c]{cl}
& \Delta_{B}R_{ijkl}\\
= & \sum_{p=1}^{2n}R_{ijkl,p}\psi_{p}+4nR_{ijkl}+\sum_{p=1}^{2n}R_{klpi}%
R_{pj}-\sum_{p=1}^{2n}R_{klpj}R_{pi}\\
& +2\sum_{p,q=1}^{2n}R_{qlpj}R_{qkip}+2\sum_{p,q=1}^{2n}R_{kqpj}R_{qlip}%
+\sum_{p,q=1}^{2n}R_{klpq}R_{qpij}\\
& +\sum_{q=1}^{2n}R_{klqj}(R_{qi}-g_{qi})-\sum_{q=1}^{2n}R_{klqi}%
(R_{qj}-g_{qj}).
\end{array}
\]
Notice that%
\[
4nR_{ijkl}-\sum_{q=1}^{2n}R_{klqj}g_{qi}+\sum_{q=1}^{2n}R_{klqi}%
g_{qj}=(4n-2)R_{ijkl}%
\]
and%
\[
\sum_{p=1}^{2n}R_{klpi}R_{pj}-\sum_{q=1}^{2n}R_{klqi}R_{qj}=0\text{
}\mathrm{and}\text{ }\sum_{q=1}^{2n}R_{klqj}R_{qi}-\sum_{p=1}^{2n}%
R_{klpj}R_{pi}=0.
\]
Thus we have%
\[%
\begin{array}
[c]{ccl}%
\Delta_{B}R_{ijkl} & = & \sum_{p=1}^{2n}R_{ijkl,p}\psi_{p}+(4n-2)R_{ijkl}\\
&  & +2\sum_{p,q=1}^{2n}(R_{iplq}R_{jpkq}-R_{ipkq}R_{jplq})-\sum_{p,q=1}%
^{2n}R_{pqij}R_{pqkl}.
\end{array}
\]
This implies%
\begin{equation}%
\begin{array}
[c]{cl}
& \Delta_{B,\psi}R_{ijkl}\\
= & \Delta_{B}R_{ijkl}-\langle \nabla \psi,\nabla R_{ijkl}\rangle \\
= & \Delta_{B}R_{ijkl}-\sum_{p=1}^{2n}R_{ijkl,p}\psi_{p}\\
= & (4n-2)R_{ijkl}+2\sum_{p,q=1}^{2n}(R_{iplq}R_{jpkq}-R_{ipkq}R_{jplq})\\
& -\sum_{p,q=1}^{2n}R_{pqij}R_{pqkl}.
\end{array}
\label{lem23}%
\end{equation}
From the last equality, we get%
\[%
\begin{array}
[c]{cl}
& \Delta_{B,\psi}\left(  \sum_{i=1}^{2n}R_{ijil}\right) \\
= & (4n-2)\sum_{i=1}^{2n}R_{ijil}+2\sum_{i,p,q=1}^{2n}(R_{iplq}R_{jpiq}%
-R_{ipiq}R_{jplq})\\
& -\sum_{i,p,q=1}^{2n}R_{pqij}R_{pqil}\\
= & (4n-2)(R_{jl}-g_{jl})+2\sum_{i,p,q=1}^{2n}R_{iplq}R_{jpiq}\\
& -2\sum_{p,q=1}^{2n}(R_{pq}-g_{pq})R_{jplq}-\sum_{i,p,q=1}^{2n}%
R_{pqij}R_{pqil}.
\end{array}
\]
By observing that%
\[
2\sum_{p,q=1}^{2n}g_{pq}R_{jplq}=2\sum_{p=1}^{2n}R_{jplp}=2(R_{jl}-g_{jl})
\]
and%
\[
\Delta_{B,\psi}R_{jl}=\Delta_{B,\psi}\left(  \sum_{i=1}^{2n}R_{ijil}%
+R_{0j0l}\right)  =\Delta_{B,\psi}\left(  \sum_{i=1}^{2n}R_{ijil}\right)  ,
\]
we obtain%
\begin{equation}%
\begin{array}
[c]{ccl}%
\Delta_{B,\psi}R_{jl} & = & 4n(R_{jl}-g_{jl})+2\sum_{i,p,q=1}^{2n}%
R_{iplq}R_{jpiq}\\
&  & -2\sum_{p,q=1}^{2n}R_{pq}R_{jplq}-\sum_{i,p,q=1}^{2n}R_{ijpq}R_{pqil}.
\end{array}
\label{lem25}%
\end{equation}
With the help of the first Bianchi identity, we obtain%
\begin{equation}%
\begin{array}
[c]{cl}
& 2\sum_{i,p,q=1}^{2n}R_{iplq}R_{jpiq}-\sum_{i,p,q=1}^{2n}R_{ijpq}R_{pqil}\\
= & 2\sum_{i,p,q=1}^{2n}R_{pilq}R_{jipq}+\sum_{i,p,q=1}^{2n}R_{jipq}R_{pqil}\\
= & \sum_{i,p,q=1}^{2n}R_{jipq}(2R_{pilq}+R_{pqil})\\
= & \sum_{i,p,q=1}^{2n}R_{jipq}(2R_{pilq}-R_{pilq}-R_{plqi})\\
= & \sum_{i,p,q=1}^{2n}(R_{jipq}R_{pilq}-R_{jipq}R_{plqi}).
\end{array}
\label{lem26}%
\end{equation}

From the equality%
\[
\sum_{i,p,q=1}^{2n}R_{jipq}R_{plqi}=\sum_{i,p,q=1}^{2n}R_{jiqp}R_{qlpi}%
=-\sum_{i,p,q=1}^{2n}R_{jipq}R_{piql},
\]
the equation (\ref{lem26}) says%
\begin{equation}
2\sum_{i,p,q=1}^{2n}R_{iplq}R_{jpiq}-\sum_{i,p,q=1}^{2n}R_{ijpq}R_{pqil}%
=\sum_{i,p,q=1}^{2n}R_{jipq}(R_{pilq}+R_{piql})=0. \label{lem27}%
\end{equation}
Combining (\ref{lem25}) with (\ref{lem27}), we get%
\[
\Delta_{B,\psi}R_{jl}=4n(R_{jl}-g_{jl})-2\sum_{p,q=1}^{2n}R_{pq}R_{jplq}.
\]
Therefore, the lemma follows.
\end{proof}

Now we are able to establish an integral bound of the Ricci curvature for any
gradient shrinking Sasaki-Ricci soliton as follows.

\begin{proposition}
\label{proposition1}If $\left(  M,g,X\right)  $ is a complete Sasaki-Ricci
soliton with the Hamiltonian potential $\psi$ with respect to $X$, then we
have
\[
\int_{M}|Ric|^{2}e^{-\lambda \psi}<\infty
\]
for any $\lambda>0$.
\end{proposition}

\begin{proof}
From the equality
\begin{equation}
\sum_{i=1}^{2n}R_{ii}=\sum_{i=1}^{2n}Ric(e_{i},e_{i})=R^{T}-4n=R-2n
\label{SR6}%
\end{equation}
and (\ref{SR9}), we obtain
\begin{equation}%
\begin{array}
[c]{ccl}%
\sum_{i,j=1}^{2n}R_{ij}^{2} & = & \sum_{i,j=1}^{2n}R_{ij}(2ng_{ij}-\psi
_{ij})\\
& = & 2n\sum_{i=1}^{2n}R_{ii}-\sum_{i,j=1}^{2n}R_{ij}\psi_{ij}\\
& = & 2n(R-2n)-\sum_{i,j=1}^{2n}R_{ij}\psi_{ij}.
\end{array}
\label{SR17}%
\end{equation}
Then by
\begin{equation}
R_{00}=Ric(\xi,\xi)=\sum_{i=1}^{2n}R(e_{i},\xi,e_{i},\xi)+R_{0000}=2n,
\label{SR4}%
\end{equation}%
\begin{equation}%
\begin{array}
[c]{ccl}%
R_{i0} & = & \sum_{\alpha=1}^{2n+1}R_{i\alpha0\alpha}=\sum_{j=1}^{2n}%
R_{ij0j}+R_{i000}\\
& = & -\sum_{j=1}^{2n}R_{j0ij}=-\sum_{j=1}^{2n}\langle R(e_{j},\xi)e_{j}%
,e_{i}\rangle \\
& = & -\sum_{j=1}^{2n}\langle g(\xi,e_{j})e_{j}-g(e_{j},e_{j})\xi,e_{i}%
\rangle \\
& = & 0,
\end{array}
\label{SR4b}%
\end{equation}
and (\ref{SR17}), the identity%
\begin{equation}%
\begin{array}
[c]{ccc}%
|Ric|^{2} & = & \sum_{i,j=1}^{2n}R_{ij}^{2}+2\sum_{i=1}^{2n}R_{i0}^{2}%
+R_{00}^{2}\\
& = & \sum_{i,j=1}^{2n}R_{ij}^{2}+4n^{2}\\
& = & 2nR-\sum_{i,j=1}^{2n}R_{ij}\psi_{ij}%
\end{array}
\label{SR18}%
\end{equation}
holds. Let $\phi$ be a cut-off function on $M$. The above equality
(\ref{SR18}) implies%
\begin{equation}%
\begin{array}
[c]{cl}
& \int_{M}|Ric|^{2}e^{-\lambda \psi}\phi^{2}\\
= & 2n\int_{M}Re^{-\lambda \psi}\phi^{2}-\sum_{i,j=1}^{2n}\int_{M}R_{ij}%
\psi_{ij}e^{-\lambda \psi}\phi^{2}\\
= & 2n\int_{M}Re^{-\lambda \psi}\phi^{2}+\sum_{i,j=1}^{2n}\int_{M}\psi
_{i}(R_{ij}e^{-\lambda \psi}\phi^{2})_{,j}\\
= & 2n\int_{M}Re^{-\lambda \psi}\phi^{2}+\sum_{i,j=1}^{2n}\int_{M}\psi
_{i}(R_{ij}e^{-\psi})_{,j}e^{(1-\lambda)\psi}\cdot \phi^{2}\\
& +\sum_{i,j=1}^{2n}\int_{M}\psi_{i}R_{ij}e^{-\lambda \psi}(1-\lambda)\psi
_{j}\phi^{2}+\sum_{i,j=1}^{2n}\int_{M}\psi_{i}R_{ij}e^{-\lambda \psi}(\phi
^{2})_{j}.
\end{array}
\label{SR19a}%
\end{equation}
By the equality (\ref{lem3}) in the proof of the lemma as precedes, we get
\[
\sum_{j-1}^{2n}(R_{jl,j}-R_{jj,l})=\sum_{i,j=1}^{2n}\psi_{i}R_{ijjl}.
\]
It follows%
\[%
\begin{array}
[c]{ccl}%
\sum_{j=1}^{2n}R_{jl,j}-(R-2n)_{,l} & = & -\sum_{i=1}^{2n}\psi_{i}%
(R_{il}-R_{i0l0})\\
& = & -\sum_{i=1}^{2n}\psi_{i}R_{il}+\psi_{l}.
\end{array}
\]
Namely,%
\begin{equation}
\sum_{j=1}^{2n}R_{jl,j}-R_{,l}=-\sum_{i=1}^{2n}\psi_{i}R_{il}+\psi_{l}.
\label{SR11}%
\end{equation}
From the second Bianchi identity, we derive that
\begin{equation}
R_{ijkj,p}+R_{ijjp,k}+R_{ijpk,j}=0. \label{SR12a}%
\end{equation}
Notice that%
\begin{equation}
\sum_{j=1}^{2n}R_{ijkj,p}=\sum_{\alpha=1}^{2n+1}R_{i\alpha k\alpha
,p}-R_{i0k0,p}=R_{ik,p} \label{SR12b1}%
\end{equation}
and%
\begin{equation}
\sum_{j=1}^{2n}R_{ijjp,k}=-\sum_{\alpha=1}^{2n+1}R_{i\alpha p\alpha
,k}+R_{i0p0,k}=-R_{ip,k}. \label{SR12b2}%
\end{equation}
Thus by (\ref{SR12a}), (\ref{SR12b1}) and (\ref{SR12b2}), we get
$R_{ik,p}-R_{ip,k}+\sum_{j=1}^{2n}R_{ijpk,j}=0$; therefore,
\[
\sum_{i=1}^{2n}R_{ii,p}-\sum_{i=1}^{2n}R_{ip,i}+\sum_{i,j=1}^{2n}%
R_{ijpi,j}=0.
\]
Namely,%
\begin{equation}
R_{,l}=\sum_{i=1}^{2n}(R_{ii}+R_{00})_{,l}=\sum_{i=1}^{2n}R_{ii,l}=2\sum
_{j=1}^{2n}R_{lj,j}. \label{SR14}%
\end{equation}
Combining (\ref{SR11}) with (\ref{SR14}), it could be observed
\begin{equation}
R_{,l}=2\sum_{i=1}^{2n}\psi_{i}R_{il}-2\psi_{l}. \label{SR15}%
\end{equation}
Hence by (\ref{SR14}) and (\ref{SR15}), we derive%
\begin{equation}%
\begin{array}
[c]{ccl}%
\sum_{i=1}^{2n}\nabla_{i}(R_{ij}e^{-\psi}) & = & \sum_{i=1}^{2n}%
R_{ij,i}e^{-\psi}-\sum_{i=1}^{2n}R_{ij}e^{-\psi}\psi_{i}\\
& = & e^{-\psi}\left(  \sum_{i=1}^{2n}R_{ij,i}-\sum_{i=1}^{2n}\psi_{i}%
R_{ij}\right) \\
& = & e^{-\psi}(\frac{1}{2}R_{,j}-\sum_{i=1}^{2n}\psi_{i}R_{ij})\\
& = & -e^{-\psi}\psi_{j}.
\end{array}
\label{SR16}%
\end{equation}
Sustituting this equality into (\ref{SR19a}), we derive%
\begin{equation}%
\begin{array}
[c]{cl}
& \int_{M}|Ric|^{2}e^{-\lambda \psi}\phi^{2}\\
= & 2n\int_{M}Re^{-\lambda \psi}\phi^{2}+(1-\lambda)\sum_{i,j=1}^{2n}\int
_{M}R_{ij}\psi_{i}\psi_{j}e^{-\lambda \psi}\phi^{2}\\
& -\int_{M}|\nabla \psi|^{2}e^{-\lambda \psi}\phi^{2}+\sum_{i,j=1}^{2n}\int
_{M}\psi_{i}R_{ij}e^{-\lambda \psi}(\phi^{2})_{j}.
\end{array}
\label{SR19}%
\end{equation}
By (\ref{SR6}), (\ref{SR18}) and the Cauchy-Schwarz inequality, we deduce%
\begin{equation}%
\begin{array}
[c]{cl}
& 2n\int_{M}Re^{-\lambda \psi}\phi^{2}\\
= & 2n\int_{M}(\sum_{i}R_{ii}+2n)e^{-\lambda \psi}\phi^{2}\\
\leq & \frac{1}{8n}\int_{M}(\sum_{i}R_{ii})^{2}e^{-\lambda \psi}\phi^{2}%
+8n^{3}\int_{M}e^{-\lambda \psi}\phi^{2}+4n^{2}\int_{M}e^{-\lambda \psi}\phi
^{2}\\
\leq & \frac{1}{4}\int_{M}|Ric|^{2}e^{-\lambda \psi}\phi^{2}+(8n+3)n^{2}%
\int_{M}e^{-\lambda \psi}\phi^{2}.
\end{array}
\label{SR21a}%
\end{equation}
An easy algebraic manipulation gives us%
\begin{equation}%
\begin{array}
[c]{cl}
& (1-\lambda)\sum_{i,j}\int_{M}R_{ij}\psi_{i}\psi_{j}e^{-\lambda \psi}\phi
^{2}\\
\leq & \frac{1}{4}\int_{M}(|Ric|^{2}-4n^{2})e^{-\lambda \psi}\phi
^{2}+(1-\lambda)^{2}\int_{M}|\nabla \psi|^{4}e^{-\lambda \psi}\phi^{2}%
\end{array}
\label{SR21}%
\end{equation}
and%
\begin{equation}%
\begin{array}
[c]{cl}
& \sum_{i,j}\int_{M}\psi_{i}R_{ij}e^{-\lambda \psi}(\phi^{2})_{j}\\
\leq & \frac{1}{4}\int_{M}(|Ric|^{2}-4n^{2})e^{-\lambda \psi}\phi^{2}+4\int
_{M}|\nabla \psi|^{2}e^{-\lambda \psi}|\nabla \phi|^{2}.
\end{array}
\label{SR21c}%
\end{equation}
From the equations (\ref{SR19}) to (\ref{SR21c}), we have%
\begin{equation}%
\begin{array}
[c]{cl}
& \int_{M}|Ric|^{2}e^{-\lambda \psi}\phi^{2}\\
\leq & 4(8n+1)n^{2}\int_{M}e^{-\lambda \psi}\phi^{2}+4(1-\lambda)^{2}\int
_{M}|\nabla \psi|^{4}e^{-\lambda \psi}\phi^{2}\\
& +16\int_{M}|\nabla \psi|^{2}e^{-\lambda \psi}|\nabla \phi|^{2}.
\end{array}
\label{SR21d}%
\end{equation}
It follows from Lemma \ref{potential estimate} by taking $C_{1}=0$, we have%
\begin{equation}
R+|\nabla \psi|^{2}=(4n-2)\psi. \label{SR22}%
\end{equation}
By Proposition \ref{prop1}, the potential estimate%
\begin{equation}
n\left(  d\left(  x,y\right)  -7\right)  _{+}^{2}\leq \psi \left(  x\right)
\leq n\left(  d\left(  x,y\right)  +\sqrt{3}\right)  ^{2} \label{SR23}%
\end{equation}
holds where $y$ is a minimum point of $\psi$. From the potential estimate, for
any $\mu>0$, it gives the inequality%
\begin{equation}%
\begin{array}
[c]{ccl}%
\int_{M}e^{-\mu \psi} & = & \sum_{j=0}^{\infty}\int_{B_{(j+1)r}\backslash
B_{jr}}e^{-\mu \psi}\\
& \leq & \sum_{j=0}^{\infty}e^{-\mu n(jr-7)^{2}}\cdot \mathrm{Vol}%
(B_{(j+1)r})\\
& \leq & \sum_{j=0}^{\infty}Ce^{-\mu n(jr-7)^{2}}(j+1)^{d}r^{d}<\infty.
\end{array}
\label{SGRS11}%
\end{equation}
The formula (6.1)(v) in \cite{cll25} implies that the scalar curvature $R$
satisfies%
\begin{equation}
R\geq \frac{C_{3}}{\psi} \label{SR24}%
\end{equation}
for some positive constant $C_{3}$. From (\ref{SR22})-(\ref{SR24}), we find
that%
\[
\int_{M}|\nabla \psi|^{4}e^{-\lambda \psi}<\infty \quad \mathrm{and}\quad \int
_{M}|\nabla \psi|^{2}e^{-\lambda \psi}<\infty.
\]
Hence by (\ref{SR21d}), we conclude that%
\[
\int_{M}|Ric|^{2}e^{-\lambda \psi}<\infty.
\]
This completes the proof.
\end{proof}

\begin{proposition}
\label{proposition2}Under the assumption that $\left(  M,g,X\right)  $ is a
complete Sasaki-Ricci soliton with the Hamiltonian potential $\psi$ with
respect to $X$ and $\int_{M}|Rm|^{2}e^{-\lambda \psi}<\infty$ for some
$\lambda \in \left(  0,1\right)  $, then the following estimate holds%
\[
\int_{M}|{\nabla Ric}|^{2}e^{-\psi}=\int_{M}|\operatorname{div}^{T}%
\,Rm|^{2}e^{-\psi}<\infty.
\]

\end{proposition}

\begin{proof}
By (\ref{lem3}), we obtain%
\begin{equation}
\sum_{i=1}^{2n}\nabla_{i}\left(  R_{ijkl}e^{-\psi}\right)  =e^{-\psi}\left(
\sum_{i=1}^{2n}R_{ijkl,i}-\sum_{i=1}^{2n}\psi_{i}R_{ijkl}\right)  =0.
\label{SRS3}%
\end{equation}
The second Bianchi identity implies that%
\begin{equation}%
\begin{array}
[c]{ccl}%
R_{ij,0} & = & \sum_{\alpha=1}^{2n+1}R_{i\alpha j\alpha,0}=\sum_{k=1}%
^{2n}R_{ikjk,0}+R_{i0j0,0}\\
& = & \sum_{k=1}^{2n}R_{ikjk,0}=-\sum_{k=1}^{2n}(R_{ikk0,j}+R_{ik0j,k})\\
& = & 0.
\end{array}
\label{SRS3b}%
\end{equation}
Combining (\ref{SR4}), (\ref{SR4b}), (\ref{SR9}) with (\ref{SRS3b}), it
follows%
\begin{equation}%
\begin{array}
[c]{cl}
& |\nabla Ric|^{2}=\sum_{\alpha,\beta,\gamma=1}^{2n+1}|\nabla_{\gamma
}R_{\alpha \beta}|^{2}\\
= & \sum_{\gamma=1}^{2n+1}\sum_{i,j=1}^{2n}|\nabla_{\gamma}R_{ij}|^{2}%
+2\sum_{\gamma=1}^{2n+1}\sum_{i=1}^{2n}|\nabla_{\gamma}R_{i0}|^{2}%
+\sum_{\gamma=1}^{2n+1}|\nabla_{\gamma}R_{00}|^{2}\\
= & \sum_{\gamma=1}^{2n+1}\sum_{i,j=1}^{2n}|\nabla_{\gamma}R_{ij}|^{2}%
=\sum_{i,j,k=1}^{2n}|\nabla_{k}R_{ij}|^{2}=\sum_{i,j=1}^{2n}|\nabla^{T}%
R_{ij}|^{2}.
\end{array}
\label{SRS5}%
\end{equation}
Let $B_{r}$ be the closed geodesic ball of $M$. Let $\phi$ be a smooth cut-off
function on $M$ such that $\phi=1$ on $B_{r},\phi=0$ outside $B_{2r}$ and
$|\nabla \phi \leq \frac{C}{r}$ on $B_{2r}\backslash B_{r}$. Notice that%
\[
\Delta_{B,\psi}R_{ij}=\Delta_{B}R_{ij}-\langle \nabla \psi,\nabla R_{ij}%
\rangle=e^{\psi}\operatorname{div}^{T}\left(  e^{-\psi}\nabla^{T}%
R_{ij}\right)
\]
and%
\begin{equation}%
\begin{array}
[c]{cl}
& \operatorname{div}^{T}\left(  e^{-\psi}\nabla^{T}R_{ij}\cdot R_{ij}\phi
^{2}\right) \\
= & \operatorname{div}^{T}\left(  e^{-\psi}\nabla^{T}R_{ij}\right)  R_{ij}%
\phi^{2}+\langle e^{-\psi}\nabla^{T}R_{ij},\nabla^{T}(R_{ij}\phi^{2})\rangle \\
= & (\Delta_{B,\psi}R_{ij})R_{ij}e^{-\psi}\phi^{2}+e^{-\psi}|\nabla^{T}%
R_{ij}|^{2}\phi^{2}+e^{-\psi}R_{ij}\langle \nabla^{T}R_{ij},\nabla^{T}\phi
^{2}\rangle.
\end{array}
\label{SRS7}%
\end{equation}
Then, from (\ref{SRS5}) and (\ref{SRS7}), we have%
\[
\int_{M}|\nabla Ric|^{2}e^{-\psi}\phi^{2}=-\sum_{i,j=1}^{2n}\int_{M}%
(\Delta_{B,\psi}R_{ij})R_{ij}e^{-\psi}\phi^{2}-\sum_{i,j,k=1}^{2n}\int
_{M}(\nabla_{k}R_{ij})R_{ij}e^{-\psi}(\phi^{2})_{k}.
\]
Hence by the equality (\ref{lem1}), we obtain%
\begin{equation}%
\begin{array}
[c]{cl}
& \int_{M}|\nabla Ric|^{2}e^{-\psi}\phi^{2}\\
= & -\sum_{i,j=1}^{2n}\int_{M}\left(  4n(R_{ij}-g_{ij})-2\sum_{p,q=1}%
^{2n}R_{pq}R_{ipjq}\right)  R_{ij}e^{-\psi}\phi^{2}\\
& -\sum_{i,j,k=1}^{2n}\int_{M}(\nabla_{k}R_{ij})R_{ij}e^{-\psi}(\phi^{2}%
)_{k}\\
= & -4n\sum_{i,j=1}^{2n}\int_{M}R_{ij}^{2}e^{-\psi}\phi^{2}+4n\sum
_{i,j=1}^{2n}\int_{M}g_{ij}R_{ij}e^{-\psi}\phi^{2}\\
& +2\sum_{i,j,p,q=1}^{2n}\int_{M}R_{ipjq}R_{ij}R_{pq}e^{-\psi}\phi^{2}%
-\sum_{i,j,k=1}^{2n}\int_{M}(\nabla_{k}R_{ij})R_{ij}e^{-\psi}(\phi^{2})_{k}.
\end{array}
\label{SRS8}%
\end{equation}
The equalities (\ref{SR3}) and (\ref{SR9}) imply that%
\[%
\begin{array}
[c]{cl}
& \sum_{i,j,p,q=1}^{2n}R_{ipjq}R_{ij}R_{pq}\\
= & \sum_{i,j,p,q=1}^{2n}R_{ipjq}R_{ij}(2ng_{pq}-\psi_{pq})\\
= & 2n\sum_{i,j,p=1}^{2n}R_{ij}R_{ipjp}-\sum_{i,j,p,q=1}^{2n}R_{ipjq}%
R_{ij}\psi_{pq}\\
= & 2n\sum_{i,j=1}^{2n}R_{ij}(R_{ij}-g_{ij})-\sum_{i,j,p,q=1}^{2n}%
R_{ipjq}R_{ij}\psi_{pq}\\
= & 2n\sum_{i,j=1}^{2n}R_{ij}^{2}-2n\sum_{i,j=1}^{2n}g_{ij}R_{ij}%
-\sum_{i,j,p,q=1}^{2n}R_{ipjq}R_{ij}\psi_{pq}.
\end{array}
\]
It follows%
\begin{equation}%
\begin{array}
[c]{cl}
& 2\sum_{i,j,p,q=1}^{2n}\int_{M}R_{ipjq}R_{ij}R_{pq}e^{-\psi}\phi^{2}\\
= & 4n\sum_{i,j=1}^{2n}\int_{M}R_{ij}^{2}e^{-\psi}\phi^{2}-4n\sum_{i,j=1}%
^{2n}\int_{M}g_{ij}R_{ij}e^{-\psi}\phi^{2}\\
& -2\sum_{i,j,p,q=1}^{2n}\int_{M}R_{ipjq}R_{ij}\psi_{pq}e^{-\psi}\phi^{2}.
\end{array}
\label{SRS9}%
\end{equation}
From (\ref{SRS8}) and (\ref{SRS9}), we derive that%
\begin{equation}%
\begin{array}
[c]{ccl}%
\int_{M}|\nabla Ric|^{2}e^{-\psi}\phi^{2} & = & -2\sum_{i,j,p,q=1}^{2n}%
\int_{M}R_{ipjq}R_{ij}\psi_{pq}e^{-\psi}\phi^{2}\\
&  & -\sum_{i,j,k=1}^{2n}\int_{M}(\nabla_{k}R_{ij})R_{ij}e^{-\psi}(\phi
^{2})_{k}.
\end{array}
\label{SRS10}%
\end{equation}
Integrating by parts and the equality (\ref{SRS3}) give us%
\begin{equation}%
\begin{array}
[c]{cl}
& -2\sum_{i,j,p,q=1}^{2n}\int_{M}R_{ipjq}R_{ij}\psi_{pq}e^{-\psi}\phi^{2}\\
= & 2\sum_{i,j,p,q=1}^{2n}\int_{M}\nabla_{q}(R_{qjpi}e^{-\psi})R_{ij}\phi
^{2}\psi_{p}+2\sum_{i,j,p,q=1}^{2n}\int_{M}\nabla_{q}(R_{ij}\phi^{2}%
)R_{ipjq}e^{-\psi}\psi_{p}\\
= & 2\sum_{i,j,p,q=1}^{2n}\int_{M}R_{ipjq}(\nabla_{q}R_{ij})\psi_{p}e^{-\psi
}\phi^{2}+2\sum_{i,j,p,q=1}^{2n}\int_{M}R_{ipjq}R_{ij}\psi_{p}e^{-\psi}%
(\phi^{2})_{q}.
\end{array}
\label{SRS11}%
\end{equation}
By the equality (\ref{lem3}), we get%
\begin{equation}%
\begin{array}
[c]{cl}
& \left(  \operatorname{div}^{T}Rm\right)  _{\beta \gamma \delta}=\sum
_{\alpha=1}^{2n+1}R_{\alpha \beta \gamma \delta,\alpha}=\sum_{i=1}^{2n}%
R_{i\beta \gamma \delta,i}+R_{0\beta \gamma \delta,0}\\
= & \sum_{i=1}^{2n}R_{i\beta \gamma \delta,i}=\sum_{i=1}^{2n}R_{ijkl,i}%
=(\operatorname{div}^{T}Rm)_{jkl}=\sum_{i=1}^{2n}\psi_{i}R_{ijkl}.
\end{array}
\label{SGRS9}%
\end{equation}
The direct computation gives%
\begin{equation}%
\begin{array}
[c]{cl}
& 2\sum_{i,j,p,q=1}^{2n}R_{ipjq}(\nabla_{q}R_{ij})\psi_{p}\\
= & -2\sum_{i,j,p,q=1}^{2n}R_{qjip}\psi_{p}(\nabla_{q}R_{ij})\\
= & -2\sum_{i,j,p,q=1}^{2n}R_{jqip}\psi_{p}(\nabla_{j}R_{iq})\\
= & 2\sum_{i,j,p,q=1}^{2n}R_{qjip}\psi_{p}(\nabla_{j}R_{iq})\\
= & \sum_{i,j,p,q=1}^{2n}R_{qjip}\psi_{p}(\nabla_{j}R_{iq}-\nabla_{q}R_{ij}).
\end{array}
\label{SRS12}%
\end{equation}
Then from the equations (\ref{lem3}), (\ref{SGRS9}) and (\ref{SRS12}), we
deduce%
\begin{equation}%
\begin{array}
[c]{cl}
& 2\sum_{i,j,p,q=1}^{2n}R_{ipjq}(\nabla_{q}R_{ij})\psi_{p}\\
= & \sum_{i,j,p,q,h=1}^{2n}R_{qjip}\psi_{p}\psi_{h}R_{hijq}\\
= & \sum_{i,j,q=1}^{2n}\left(  (\operatorname{div}^{T}Rm)_{ijq}\right)  ^{2}\\
= & \sum_{\beta,\gamma,\delta=1}^{2n+1}((\operatorname{div}^{T}\,Rm)_{\beta
\gamma \delta})^{2}\\
= & |\operatorname{div}^{T}Rm|^{2}.
\end{array}
\label{SRS12b}%
\end{equation}
Substituting the equalities (\ref{SRS11}) and (\ref{SRS12b}) into the equation
(\ref{SRS10}), it could be derived that%
\begin{equation}%
\begin{array}
[c]{cl}
& \int_{M}|\nabla Ric|^{2}e^{-\psi}\phi^{2}\\
= & \int_{M}|\operatorname{div}^{T}Rm|^{2}e^{-\psi}\phi^{2}+2\sum
_{i,j,p,q=1}^{2n}\int_{M}R_{ipjq}R_{ij}\psi_{p}e^{-\psi}(\phi^{2})_{q}\\
& -\sum_{i,j,k=1}^{2n}\int_{M}(\nabla_{k}R_{ij})R_{ij}e^{-\psi}(\phi^{2})_{k}.
\end{array}
\label{SRS13}%
\end{equation}
Notice that%
\[%
\begin{array}
[c]{cl}
& \int_{M}|\operatorname{div}^{T}Rm|^{2}e^{-\psi}\phi^{2}\\
\leq & C\int_{M}|Rm|^{2}|\nabla \psi|^{2}e^{-\psi}\\
\leq & C\int_{M}|Rm|^{2}\psi e^{-\psi}\\
\leq & C\int_{M}|Rm|^{2}e^{-\lambda \psi}<\infty
\end{array}
\]
and%
\[%
\begin{array}
[c]{cl}
& 2\sum_{i,j,p,q=1}^{2n}\int_{M}R_{ipjq}R_{ij}\psi_{p}e^{-\psi}(\phi^{2}%
)_{q}\\
\leq & C\int_{M}|Rm|^{2}|\nabla \psi|e^{-\psi}\\
\leq & C\int_{M}|Rm|^{2}e^{-\lambda \psi}<\infty.
\end{array}
\]
Since%
\[%
\begin{array}
[c]{ccl}%
\int_{M}|\nabla Ric|^{2}e^{-\psi}\phi^{2} & \leq & C-\sum_{i,j,k=1}^{2n}%
\int_{M}(\nabla_{k}R_{ij})R_{ij}e^{-\psi}(\phi^{2})_{k}\\
& \leq & C+2\sum_{i,j,k=1}^{2n}\int_{M}|\nabla_{k}R_{ij}||R_{ij}|e^{-\psi}%
\phi|\nabla \phi|\\
& \leq & C+\frac{1}{2}\int_{M}|\nabla Ric|^{2}e^{-\psi}\phi^{2}+2\int
_{M}|Ric|^{2}e^{-\psi}|\nabla \phi|^{2},
\end{array}
\]
therefore by using Proposition \ref{proposition1}, we can conclude that
$\int_{M}|\nabla Ric|^{2}e^{-\psi}<\infty$. By H\"{o}lder inequality, we
derive that as $r\rightarrow \infty$%
\begin{equation}%
\begin{array}
[c]{cl}
& \left \vert \sum_{i,j,k=1}^{2n}\int_{M}(\nabla_{k}R_{ij})R_{ij}e^{-\psi}%
(\phi^{2})_{k}\right \vert \\
\leq & \frac{C}{r}\left(  \int_{M}|\nabla Ric|^{2}e^{-\psi}\right)  ^{\frac
{1}{2}}\left(  \int_{B_{2r}\backslash B_{r}}|Ric|^{2}e^{-\psi}\right)
^{\frac{1}{2}}\rightarrow0
\end{array}
\label{SRS19}%
\end{equation}
and%
\begin{equation}
\left \vert \sum_{i,j,p,q=1}^{2n}\int_{M}R_{ipjq}R_{ij}\psi_{p}e^{-\psi}%
(\phi^{2})_{q}\right \vert \leq \frac{C}{r}\int_{B_{2r}\backslash B_{r}}%
|Rm|^{2}e^{-\lambda \psi}\rightarrow0. \label{SRS20}%
\end{equation}
Combining these inequalities (\ref{SRS13}), (\ref{SRS19}) with (\ref{SRS20}),
it follows%
\[
\int_{M}|{\nabla Ric}|^{2}e^{-\psi}=\int_{M}|\operatorname{div}^{T}%
Rm|^{2}e^{-\psi}<\infty.
\]

\end{proof}

By the same reasoning, we then obtain the following corollary.

\begin{corollary}
If $\left(  M,g,X\right)  $ is a complete Sasaki-Ricci soliton with the
Hamiltonian potential $\psi$ with respect to $X$ and it admits harmonic Weyl
tensor, then we have%
\begin{equation}
\int_{M}|{\nabla Ric}|^{2}e^{-\psi}=\int_{M}|\operatorname{div}^{T}%
\,Rm|^{2}e^{-\psi}<\infty. \label{SR20}%
\end{equation}

\end{corollary}

\begin{proof}
For any $X,Y,Z\in \Gamma(D)$, we have%
\begin{equation}%
\begin{array}
[c]{cl}
& (\nabla_{X}\mathrm{Hess}\left(  \psi \right)  )(Y,Z)-(\nabla_{Y}%
\mathrm{Hess}\left(  \psi \right)  )(X,Z)\\
= & \nabla_{X}(\mathrm{Hess}\left(  \psi \right)  (Y,Z))-\nabla_{Y}%
(\mathrm{Hess}\left(  \psi \right)  (X,Z))\\
& +\mathrm{Hess}\left(  \psi \right)  (\nabla_{Y}X-\nabla_{X}Y,Z)\\
& +\mathrm{Hess}\left(  \psi \right)  (X,\nabla_{Y}Z)-\mathrm{Hess}\left(
\psi \right)  (Y,\nabla_{X}Z).
\end{array}
\label{SGRS1}%
\end{equation}
Notice that%
\begin{equation}
\left \{
\begin{array}
[c]{l}%
\nabla_{X}(\mathrm{Hess}\left(  \psi \right)  (Y,Z))=\nabla_{X}\langle
\nabla_{Y}\nabla \psi,Z\rangle=\langle \nabla_{X}\nabla_{Y}\nabla \psi
,Z\rangle+\langle \nabla_{Y}\nabla \psi,\nabla_{X}Z\rangle,\\
\nabla_{Y}(\mathrm{Hess}\left(  \psi \right)  (X,Z))=\nabla_{Y}\langle
\nabla_{X}\nabla \psi,Z\rangle=\langle \nabla_{Y}\nabla_{X}\nabla \psi
,Z\rangle+\langle \nabla_{X}\nabla \psi,\nabla_{Y}Z\rangle
\end{array}
\right.  \label{SGRS2}%
\end{equation}
and%
\begin{equation}
\psi(X,\nabla_{Y}Z)-\mathrm{Hess}\left(  \psi \right)  (Y,\nabla_{X}%
Z)=\langle \nabla_{X}\nabla \psi,\nabla_{Y}Z\rangle-\langle \nabla_{Y}\nabla
\psi,\nabla_{X}Z\rangle. \label{SGRS3}%
\end{equation}
Substituting the identities (\ref{SGRS2}) and (\ref{SGRS3}) into
(\ref{SGRS1}), we get%
\begin{equation}
(\nabla_{X}\mathrm{Hess}\left(  \psi \right)  )(Y,Z)-(\nabla_{Y}\mathrm{Hess}%
\left(  \psi \right)  )(X,Z)=\langle R(X,Y)\nabla \psi,Z\rangle. \label{SGRS4}%
\end{equation}
If the Weyl tensor is harmonic, then the Schouten tensor $S\doteqdot
Ric-\frac{R}{4n}g$ is a Codazzi tensor. It yields that $(\nabla_{X}%
S)(Y,Z)=(\nabla_{Y}S)(X,Z)$. By noting that $S=Ric-\frac{R}{4n}g=Ric^{T}%
-2g-\frac{R}{4n}g$, we derive%
\[
\nabla_{X}S=\nabla_{X}Ric^{T}-\frac{\nabla_{X}R}{4n}g.
\]
It follows that%
\begin{equation}
(\nabla_{X}Ric^{T})(Y,Z)-(\nabla_{Y}Ric^{T})(X,Z)=\frac{\nabla_{X}R}%
{4n}g(Y,Z)-\frac{\nabla_{Y}R}{4n}g(X,Z). \label{SGRS5}%
\end{equation}
Combining equations (\ref{SGRS4}) and (\ref{SGRS5}) with the Sasaki-Ricci
soliton equation, we have%
\begin{equation}%
\begin{array}
[c]{cl}
& (2n+2)(\nabla_{X}g^{T})(Y,Z)-(2n+2)(\nabla_{Y}g^{T})(X,Z)\\
= & \frac{\nabla_{X}R}{4n}g(Y,Z)-\frac{\nabla_{Y}R}{4n}g(X,Z)+\langle
R(X,Y)\nabla \psi,Z\rangle.
\end{array}
\label{SGRS6}%
\end{equation}
A direct computation yields%
\[
\nabla_{X}g^{T}=\nabla_{X}g-\nabla_{X}(\eta \otimes \eta)=0.
\]
Therefore the equation (\ref{SGRS6}) becomes%
\begin{equation}
R(X,Y,Z,\nabla \psi)=\langle R(X,Y)\nabla \psi,Z\rangle=\frac{\nabla_{Y}R}%
{4n}g(X,Z)-\frac{\nabla_{X}R}{4n}g(Y,Z). \label{SGRS7}%
\end{equation}
Using the identity (\ref{SR15}), the last equation can be written as
\begin{equation}%
\begin{array}
[c]{cl}
& R(X,Y,Z,\nabla \psi)\\
= & \frac{Ric(Y,\nabla \psi)-g(Y,\nabla \psi)}{2n}g(X,Z)-\frac{Ric(X,\nabla
\psi)-g(X,\nabla \psi)}{2n}g(Y,Z).
\end{array}
\label{SGRS8}%
\end{equation}
Taking $Z=\nabla \psi$, we obtain%
\[
Ric(Y,\nabla \psi)g(X,\nabla \psi)=Ric(X,\nabla \psi)g(Y,\nabla \psi).
\]
If we consider the vector field $Y$ that is perpendicular to $\nabla \psi$,
then by the equation (\ref{SR9}), for every $X\in \Gamma(D)$,%
\[
0=Ric(Y,\nabla \psi)g(X,\nabla \psi)=\mathrm{Hess}(\psi)(Y,\nabla \psi
)g(X,\nabla \psi).
\]
Thus $\nabla \psi$ is an eigenvalue of $Ric$ and $\mathrm{Hess}(\psi)$. From
(\ref{SGRS11}), (\ref{SGRS9}), (\ref{SGRS8}) and Proposition
\ref{proposition1}, we derive%
\[%
\begin{array}
[c]{ccl}%
\int_{M}|\operatorname{div}^{T}Rm|^{2}e^{-\psi} & \leq & C\int_{M}%
|Ric|^{2}|\nabla \psi|^{2}e^{-\psi}+C\int_{M}|\nabla \psi|^{2}e^{-\psi}\\
& \leq & C\int_{M}|Ric|^{2}e^{-\mu \psi}+C\int_{M}\psi e^{-\psi}\\
& \leq & C+C\int_{M}e^{-\mu \psi}<\infty
\end{array}
\]
where $\mu \in(0,1)$ is a constant. Moreover, we obtain, as $r\rightarrow
\infty$,%
\begin{equation}%
\begin{array}
[c]{cl}
& \sum_{i,j,p,q=1}^{2n}\int_{M}R_{ipjq}R_{ij}\psi_{p}e^{-\psi}(\phi^{2})_{q}\\
\leq & \frac{c}{r}\left(  \int_{M}|\operatorname{div}^{T}Rm|^{2}e^{-\psi}%
+\int_{M}|Ric|^{2}e^{-\psi}\right) \\
\leq & \frac{C}{r}\rightarrow0,
\end{array}
\label{SGRS13}%
\end{equation}
then, combining (\ref{SRS13}), (\ref{SRS19}) with (\ref{SGRS13}), then the
proof is completed.
\end{proof}

Now we are in a position to prove Theorem \ref{harWT} as follows.

\begin{proof}
(Proof of Theorem \ref{harWT}) Let $\{E_{1},\cdot \cdot \cdot,E_{2n}\}
\subset \Gamma(D)$ be the eigenvectors of $Ric$ with $g\left(  E_{i}%
,E_{j}\right)  =\delta_{ij}$ and $E_{2n}=\frac{\nabla \psi}{|\nabla \psi|}$.
Then employing the equations (\ref{lem3}) and (\ref{SGRS8}), we get%
\begin{equation}%
\begin{array}
[c]{ccl}%
\, \left \vert Rm\right \vert ^{2} & = & \sum_{j,k,l=1}^{2n}|(\operatorname{div}%
^{T}Rm)(E_{j},E_{k},E_{l})|^{2}\\
& = & \sum_{j,k,l=1}^{2n}|R(\nabla \psi,E_{j},E_{k},E_{l})|^{2}\\
& = & \sum_{k,l=1}^{2n}|R(\nabla \psi,E_{k},E_{k},E_{l})|^{2}+\sum_{k,l=1}%
^{2n}|R(\nabla \psi,E_{l},E_{k},E_{l})|^{2}\\
&  & +\sum_{k,l=1}^{2n}\sum_{j\neq k,j\neq l}|R(\nabla \psi,E_{j},E_{k}%
,E_{l})|^{2}\\
& = & 2\sum_{k,l=1}^{2n}|R(\nabla \psi,E_{k},E_{k},E_{l})|^{2}\\
& = & \frac{1}{n}\sum_{l=1}^{2n}|Ric(E_{l},\nabla \psi)-g(E_{l},\nabla
\psi)|^{2}\\
& = & \frac{1}{4n}|\nabla R|^{2}.
\end{array}
\label{SGRS14}%
\end{equation}
Applying the Cauchy-Schwarz inequality and the equation (\ref{SRS5}), we
obtain%
\begin{equation}%
\begin{array}
[c]{ccl}%
|\nabla Ric|^{2} & = & \sum_{i,j,k=1}^{2n}(R_{ij,k})^{2}\\
& \geq & \sum_{i,k=1}^{2n}(R_{ii,k})^{2}\\
& \geq & \frac{1}{2n}\sum_{k=1}^{2n}\left(  \sum_{i=1}^{2n}R_{ii,k}\right)
^{2}\\
& = & \frac{1}{2n}\sum_{k=1}^{2n}\left[  (R-2n)_{,k}\right]  ^{2}\\
& = & \frac{1}{2n}|\nabla R|^{2}.
\end{array}
\label{SGRS15}%
\end{equation}
From the equalities (\ref{SR20}), (\ref{SGRS14}) and the inequality
(\ref{SGRS15}), we deduce the estimate%
\[
\frac{1}{4n}\int_{M}|\nabla R|^{2}e^{-\psi}\geq \frac{1}{2n}\int_{M}|\nabla
R|^{2}e^{-\psi}.
\]
This forces that the scalar curvature $R$ is constant. Therefore by the
equation (\ref{SR15}), we derive that%
\begin{equation}
Ric(\nabla \psi,\nabla \psi)-|\nabla \psi|^{2}=\frac{1}{2}\langle \nabla
R,\nabla \psi \rangle=0. \label{SGRS16}%
\end{equation}
From the identities (\ref{SGRS8}) and (\ref{SGRS16}), we get that the
transversely radial curvature vanishes as below:
\[%
\begin{array}
[c]{ccl}%
\kappa_{rad}^{T} & = & \sum_{i=1}^{2n}R(E_{i},\nabla \psi,E_{i},\nabla \psi)\\
& = & \sum_{i=1}^{2n-1}\frac{1}{2n}\left[  Ric(\nabla \psi,\nabla \psi
)-|\nabla \psi|^{2}\right]  g(E_{i},E_{i})\\
& = & \frac{2n-1}{2n}\left[  Ric(\nabla \psi,\nabla \psi)-|\nabla \psi
|^{2}\right] \\
& = & 0.
\end{array}
\]
Then by Theorem 1 in \cite{cll25}, $(M,g,\psi)$ is actually transversely
rigid. Therefore $M$ is Sasaki-Einstein and a compact gradient shrinking Ricci
soliton. Then by Theorem 2.5 in \cite{ms13}, $M$ is a finite quotient of the
sphere $\mathbb{S}^{2n+1}$.
\end{proof}

\section{Statements and Declarations}

\textbf{Data availability }Data sharing not applicable to this article as no
datasets were generated or analysed during the current study.

\textbf{Conflict of interest} The authors have no financial or personal
relationships with other people or organizations that can inappropriately
influence the submitted work, there is no professional or other personal
interest of any nature or kind in any product, service and/or company that
could be construed as influencing the position presented in, or the review of,
the manuscript entitled.

\textbf{Ethics approval} The manuscript has not been submitted to more than
one journal for simultaneous consideration. The submitted work is original and
has not been published elsewhere in any form or language (partially or in full).

\end{document}